\documentclass[10pt, reqno]{amsart}

\newtheorem{theorem}{Theorem}[section]

\newtheorem{lem}[theorem]{Lemma}
\newtheorem{cor}[theorem]{Corollary}

\theoremstyle{definition}
\newtheorem{de}[theorem]{Definition}
\theoremstyle{remark}
\newtheorem{re}[theorem]{Remark}
\numberwithin{equation}{section}
\usepackage{xcolor}
\usepackage{amsmath}
\usepackage{amssymb}
\usepackage{amsthm}
\usepackage[utf8]{inputenc}
\usepackage[normalem]{ulem}
\usepackage{hyperref}
\usepackage{mathtools}
\mathtoolsset{showonlyrefs}

\usepackage[top=3cm, bottom=2cm, left=3cm, right=3cm]{geometry}
\usepackage[abbrev]{amsrefs}
\usepackage{bbm}
\renewcommand{\MR}[1]{}

\usepackage{slashed}
\newcommand*{\rom}[1]{\expandafter\@slowromancap\romannumeral #1@}
\begingroup 
\makeatletter
   \@for\theoremstyle:=theorem,de,re,pro,lem,coro,plain\do{%
     \expandafter\g@addto@macro\csname th@\theoremstyle\endcsname{%
        \addtolength\thm@preskip\parskip
     }%
   }
\endgroup

\begin{document}

\title[]{Sharp well-posedness and ill-posedness results for the inhomogeneous NLS equation}



\author[L. Campos]{Luccas Campos}
\address{Department of Mathematics, UFMG,
Av. Pres. Antônio Carlos, 6627, Pampulha, 31270-901, Belo Horizonte, Minas Gerais, Brazil}
\email{luccas@mat.ufmg.br}

\author[S. Correia]{Sim\~ao Correia}
\address{
        Center for Mathematical Analysis, Geometry and Dynamical Systems,
        Department of Mathematics,
        Instituto Superior T\'ecnico, Universidade de Lisboa,
        Av. Rovisco Pais, 1049-001 Lisboa, Portugal.}
\email{simao.f.correia@tecnico.ulisboa.pt}

\author[L. G. Farah]{Luiz Gustavo Farah}
\address{Department of Mathematics, UFMG,
Av. Pres. Antônio Carlos, 6627, Pampulha, 31270-901, Belo Horizonte, Minas Gerais, Brazil}
\email{farah@mat.ufmg.br}


\begin{abstract}
We consider the initial value problem associated to the inhomogeneous nonlinear Schrö\-din\-ger equation,
		\begin{equation}
			iu_t + \Delta u +\mu|x|^{-b}|u|^{\alpha}u=0, \quad u_0\in H^s(\mathbb R^N) \text{ or } u_0 \in\dot H ^s(\mathbb R^N),
		\end{equation}
with $\mu=\pm 1$, $b > 0$, $s\geq 0$ and $0 < \alpha \leq \frac{4-2b}{N-2s}$. By means of an adapted version of the fractional Leibniz rule, we prove new local well-posedness results in Sobolev spaces for a large range of parameters. We also prove an ill-posedness result for this equation, through a delicate analysis of the associated Duhamel operator.
\end{abstract}

\maketitle

\section{Introduction}

\subsection{Setting and motivation} In this work, we consider the inhomogeneous nonlinear Schrö\-din\-ger equation

\begin{equation}\label{INLS}\tag{INLS}
iu_t + \Delta u +\mu  |x|^{-b}|u|^{\alpha}u=0,
\end{equation}
where $u: \mathbb{R}^N\times\mathbb{R} \to \mathbb{C}$, $\mu=\pm 1$, $0 < b < \min\{2,N\}$ and 
\begin{equation}\label{cond_p}
0 < \alpha \leq \alpha_s, \,\,\mbox{with}\,\, \alpha_s=
\begin{cases}
\frac{4-2b}{N-2s},& \mbox{if} \,\, s< N/2,\\
\infty,& \mbox{if} \,\, s\geq N/2.
\end{cases}
\end{equation}

This model has been a topic of intense research in the last few years (\cite{g_8}, \cite{Boa}, \cite{CHL20}, \cite{AT21}, \cite{AK21}, \cite{Boa_Dinh}, \cite{Campos_New_2019}, \cite{MMZ21}, \cite{KLS21}, \cite{LS21}). Physically, inhomogeneous NLS equations can be used to study the nonlinear propagation of laser beams subject to spatially dependent interactions (see e.g. \cite{belmonte2007lie} and the references therein). In particular, equation \eqref{INLS} can be derived as a limiting case of polynomially decaying interaction potentials (see \cite{Ge08} for more details). 

Our aim in this work is to study the well-posedness of the initial value problem (IVP) associated to equation \eqref{INLS}. First, we prove well-posedness in the usual Strichartz framework for initial data either in $H^s$ or $\dot{H}^s$ and in both subcritical and critical cases. 
We are particularly interested in addressing the fractional regularity. The natural approach, relying on the Strichartz estimates in classical Sobolev spaces, was initiated by Guzm\'an \cite{Boa}, establishing the local well-posedness in $H^s$, for $0\leq s\leq \min\{1,N/2\}$, $0<b< \min\{2,N/3\}$ and $0 < \alpha < \alpha_s$. Later, An and Kim \cite{AK21} studied the cases $0\leq s< \min\{N,1+N/2\}$, $0<b< \min\{2,N-s, N/2+1-s\}$ and $0 < \alpha < \alpha_s$ using similar ideas. In these two works, the starting point of the analysis is to split the space domain in different regions (around/far from the origin) since $|x|^{-b}$ and its fractional derivatives fail to be in any $L^p$ space, $1\leq p \leq \infty$. However, the non-local nature of the fractional derivative requires a careful treatment of the nonlinear estimates and these authors overlooked this step by inappropriately using the fractional Leibniz rule locally in space (see for instance inequalities (3.29) and (3.44) in \cite{AK21}).

Here, we develop a modification of the fractional Leibniz rule to overcome this obstacle and successfully give a complete proof of local well-posedness in $H^s$ based on Strichartz estimates in the classical Sobolev spaces. We expect this new estimate to be applicable to other problems that involve fractional derivatives locally in space.

Another question related with well-posedness is the continuous dependence on the initial data. This problem is especially difficult for small $\alpha$ and $\alpha<s<\alpha+1$. Morally, after taking $s$ derivatives on the nonlinearity, we should arrive at a term which is linear in $D^su$ and only Hölder-continuous in $u$ (of order $\alpha+1-s$). The rigorous proof of this fact has only been recently achieved in its full scope by Fujiwara \cite{F22_chain}. Combined with the localized version of Leibniz rule, this yields several Strichartz estimates for the difference of two nonlinearities which ultimately yield the continuous dependence result. 

Concerning the local well-posedness in other functional settings, under various restrictions on $b, s, N$ and $\alpha$, we refer to Aloui and Tayachi \cite{AT21} and An and Kim \cite{AK21-2}, for an approach based on Lorentz spaces, and Kim, Lee and Seo \cite{KLS21}, based on weighted $L^p$ spaces.

Finally, we prove the analytic ill-posedness of the flow for large values of $b+s$, by performing a refined analysis of the Duhamel operator. More specifically, we provide a precise asymptotic descrition of the nonlinearity when $u=e^{it\Delta}u_0$ is a free evolution. For $b+s$ outside of the local well-posedness range, the leading order term is not bounded in $H^s$, thus implying the analytic ill-posedness of the flow.

In conclusion, we are able to provide (correct) proofs of well-posedness in a \emph{sharp} range of parameters, thus settling the local well-posedness theory for the \eqref{INLS} in $H^s$ and $\dot{H}^s$ spaces.


\subsection{Statement of the results}

Recall that the critical Sobolev index is given by
\begin{equation}\label{s_c}
s_c=\frac{N}{2}-\frac{2-b}{\alpha},
\end{equation}
in the sense that the $\dot{H}^{s_c}$ norm remains invariant under the scaling $u \mapsto \lambda^{\frac{2-b}{\alpha}} u(\lambda x, \lambda^2 t)$, $\lambda>0$.

The first part of the paper is devoted to the well-posed results. Our approach is based in the fixed point method. More precisely, we want to show that the Duhamel operator
\begin{equation}\label{eq:duhamelop}
	\Phi(u)(x,t)=e^{it\Delta}u_0(x,t) + i\mu \int_0^t e^{i(t-\tau  )\Delta}|x|^{-b}|u(x,\tau)  |^{\alpha}u(x,\tau)  d\tau ,\quad t\in[0,T],
\end{equation}
where $e^{it\Delta}$ denotes the Schrödinger group, has a fixed point in a suitable complete metric space. This will follow either from a contraction mapping argument or through the application of some stability results.

We start by considering the IVP associated to \eqref{INLS} with initial data in inhomogeneous Sobolev spaces. We refer to Section \ref{Preli} for the precise definition of the auxiliary spaces mentioned in the statements below.
\begin{theorem}[Well-posedness in ${H}^s$] \label{InhTheo}
Let $N \geq 1$, $s\geq 0$, $\alpha>0$ and $0 < b <\min\{2,N-s,\frac{N}{2}+1-s \}$. Moreover, if $\alpha$ is not an even integer, assume additionally that $s<\alpha+1$. 
\begin{itemize}
\item[(a)] (Subcritical well-posedness in ${H}^s$) If 
$$
0 < \alpha <
\begin{cases}
\frac{4-2b}{N-2s},& s<N/2,\\
\infty, &s\geq N/2,
\end{cases}
$$ 
then, for any $u_0 \in {H}^s$, there exists $T = T(\|u_0\|_{{H}^s}) >0$ and a unique solution $u \in C([0,T],{H}^s) \cap  S^s(L^2,[0,T])$ to \eqref{INLS} with initial datum $u_0$.

\vskip 5pt
\item[(b)] (Critical well-posedness in ${H}^s$) If $\alpha= \frac{4-2b}{N-2s}$, then 
for any $u_0 \in {H}^s$, there exists $T = T(u_0) >0$ such that there is as unique solution $u \in C([0,T],{H}^s) \cap S^s(L^2,[0,T])$ to \eqref{INLS} with initial datum $u_0$. Moreover, if $\|u_0\|_{{H}^s}$ is sufficiently small, then the solution is global.

\end{itemize}
In both cases, if $s\le 1$, given $\delta>0$, there exists $\epsilon>0$ such that, for any $v_0\in H^s$ with $\|u_0-v_0\|_{H^s}<\epsilon$, the solution $v$ with initial data $v_0$ is defined on $[0,T]$ and satisfies
$$
\sup_{t\in[0,T]} \|u(t) - v(t)\|_{H^s} < \delta.
$$
\end{theorem}


In the subcritical case, the above result extends the one obtained by Aloui and Tayachi \cite{AT21}, by relaxing the assumption $b+2s<N$ to $b+s<N$. In the critical case, it extends the result by An and Kim \cite{AK21}, where they require the more restrictive assumption\footnote{Here $\lceil s \rceil=\min\{m\in \mathbb{N}: m\geq s\}$ and $\lfloor s \rfloor=\max\{m\in \mathbb{N}: m\leq s\}$ denote the ceiling and floor functions, respectively.} $\lceil s \rceil<\alpha+1$, if $\alpha$ is not an even integer. We also provide a detailed proof of this result using a generalized fractional Leibniz rule suitable for our context (see Lemma \ref{gen_leib_1} and Corollary \ref{gen_leib_2}). 

\begin{re}
Aloui and Tayachi \cite{AT21} proved continuous dependence under the restriction $s=0, s>0$ and $\alpha>1,$ or $\alpha=s=1$. Here, we are able to obtain the result in the full local existence range.
\end{re}


We continue our well-posedness study considering now the initial data in homogeneous Sobolev spaces with $0\leq s_c\leq s\leq 1$. 

\begin{theorem}[Well-posedness in $\dot{H}^s$]\label{HomTheo} Let $N \geq 1$, $0\leq s\leq 1$ such that $s<N/2$, 
and $0 < b <\min\{\frac{N}{2}+1-s,N-s,2\}$. 
\begin{itemize}
\item[(a)](Subcritical well-posedness in $\dot{H}^s$) If  $0 < \alpha <\frac{4-2b}{N-2s}$ and $\alpha \geq 1$ for $N=1$, then for any $u_0 \in \dot{H}^s$, there exists $T = T(\|u_0\|_{\dot{H}^s}) >0$ such that there is a unique solution $u \in C([0,T],\dot{H}^s) \cap  \dot{S}^s(L^2,[0,T])$ to \eqref{INLS} with initial datum $u_0$.
\vskip 5pt
\item[(b)] (Critical well-posedness in $\dot{H}^s$) If $\alpha =\frac{4-2b}{N-2s}$, then, for any $u_0 \in \dot{H}^s$, there exists $T = T(u_0) >0$ such that there is a unique solution $u \in C([0,T],\dot{H}^s) \cap \dot{S}^s(L^2,[0,T])$ to \eqref{INLS} with initial datum $u_0$. Moreover, if $\|u_0\|_{\dot{H}^s}$ is sufficiently small, then the solution is global.
\end{itemize}
In both cases, given $\delta>0$, there exists $\epsilon>0$ such that, for any $v_0\in H^s$ with $\|u_0-v_0\|_{\dot{H}^s}<\epsilon$, the solution $v$ with initial data $v_0$ is defined on $[0,T]$ and satisfies
$$
\sup_{t\in[0,T]} \|u(t) - v(t)\|_{\dot{H}^s} < \delta.
$$
\end{theorem}


\begin{re}
	As it is standard, the solution given in the above statements can be extended up to a maximal time of existence. In the subcritical cases, as the local time of existence depends solely of the norm of the initial data, one can prove the standard blow-up alternative. 
\end{re}

As an immediate consequence of part (a) we deduce a lower bound for the blow-up up rate of finite time solutions (see Corollary \ref{BUrate}). This rate was used as an additional assumption in \cite{MLG21} to prove the blow-up of the critical norm for the intercritical INLS equation. Our result implies that it always holds for this type of solution. The proof of part (b) relies in a different technique based in critical stability results in $\dot{H}^s$.

\begin{re}
	Although the proofs here extend in a similar fashion to the case $\max\{0,s_c\}\leq s< \frac{N}{2}$, we find that the increased technicality involved, albeit not being considerably deep, would impair the legibility of this manuscript. That said, the case $s_c< 0$, $s=0$ is already considered standard (it is treated, e.g, in \cite{g_8} and \cite{Boa}), and the case $s>1$ can be reduced to the case $0< s\leq1 $ by including up to $\lfloor s \rfloor-1$ derivatives in the corresponding norms.
\end{re}

\begin{re}
To the best of our knowledge the continuous dependence result in the homogeneous Sobolev space $\dot{H^s}$, with $0<s<1$, is new even for the classical NLS case $b=0$ (see Tao and Visan \cite{TVEJDE05} for the energy-critical case). Moveover, for the classical NLS in the inhomogeneous Sobolev setting, our techniques, based on the chain rule obtained by Fujiwara \cite{F22_chain}, provide an alternative proof to the result of Cazenave, Fang and Han \cite{CFH11,FH13} without relying on Besov spaces. 
\end{re}

\vskip10pt
In all of the above results, well-posedness was achieved under the assumptions $b+s<\min\{N, N/2+1\}$. These restrictions on $b$ and $s$ stem from the spatial norms used in the Strichartz estimates (in particular, from estimating $D^s(|x|^{-b})|u|^{\alpha}u \sim |x|^{-b-s}|u|^{\alpha}u$ in a $L^p$ space).
%
{It is hence natural to investigate whether these restrictions are indeed sharp. In this direction, we have
	\begin{theorem}[Ill-posedness]\label{thm:illposed}
		Suppose that 
		$$
		0<b<N,\quad  \min\{N,N/2+1\}\leq b+s<N/2+2, \quad \alpha\in 2\mathbb{N}\quad \mbox{and, if }b+s\ge N, \mbox{ }b+s\not\in\mathbb{N}.
		$$
	Then there does not exist a $C^{\alpha+1}$-flow in $H^s(\mathbb{R}^N)$ generated by the inhomogeneous nonlinear Schrödinger equation. The same statement holds in $\dot{H}^s(\mathbb{R}^N)$.
	\end{theorem}
	As a consequence of the above theorem, no fixed-point argument can be applied in this regime, as it would give rise to an analytic flow (see Tzvetkov \cite{tzvetkov} and Bejenaru and Tao \cite{taobejenaru} for similar results for other dispersive equations). As such, Theorems \ref{InhTheo} and \ref{HomTheo} are indeed sharp.
	\begin{re}
		We comment on the restrictions present in Theomem \ref{thm:illposed}. The restriction $0<b<N$ is imposed in order for $|x|^{-b}$ to be locally integrable. If this is not imposed, one must give sense to the nonlinear term $|x|^{-b}|u|^\alpha u$, which would change the classical notion of integral solution.
		The restriction $b+s<N/2+2$ is necessary in order to identify the main singularity of the flow (it could be removed at the expense of heavier computations). The condition $b+s\notin \mathbb{N}$ is a technical assumption related to the analytic continuation of the homogeneous distribution $|x|^{-\lambda}$ for large values of $\lambda$ (see, for example, \cite{gelfandshilov}*{Section III.3.2}).
		Finally, we require $\alpha\in 2\mathbb{N}$ in order to be able to expand the nonlinearity as powers of $u$ and $\bar{u}$. This is a crucial feature of the ill-posedness argument, see \cite{taobejenaru,tzvetkov}.
	\end{re}

\bigskip 
This paper is organized as follows. In Section \ref{Preli}, we established some preliminary estimates. The well-posed theory is discussed in Section \ref{Well}. The last section is devoted to the proofs of the ill-posedness result.

\vspace{0.5cm}
\noindent 
\textbf{Acknowledgments.} L. C. was financed by grant \#2020/10185-1, São Paulo Research Foundation (FAPESP). S.C. was partially supported by Funda\c{c}\~ao para a Ci\^encia e Tecnologia, through CAMGSD, IST-ID
(projects UIDB/04459/2020 and UIDP/04459/2020) and through the project NoDES (PTDC/
MAT-PUR/1788/2020). L.G.F. was partially supported by Coordena\c{c}\~ao de Aperfei\c{c}oamento de Pessoal de N\'ivel Superior - CAPES (project 23038.001922/2022-26), Conselho Nacional de Desenvolvimento Cient\'ifico e Tecnol\'ogico - CNPq (project 305609/2019-0) and Funda\c{c}\~ao de Amparo a Pesquisa do Estado de Minas Gerais - FAPEMIG (project PPM-00685-18).\\ 
%
\section{Preliminaries}\label{Preli}

\subsection{Notation}

Let us start this section by introducing the notation used throughout the paper. We use $C$ to denote various constants that may vary line by line. Given any positive quantities $a$ and $b$, the notation $a\lesssim b$ or $a = O(b)$ means that $a \leq Cb$, with $C$ uniform with respect to the set where $a$ and $b$ vary. If
necessary, we use subscript to indicate different parameters the constant may depends on. We also make use of the $o_\lambda(1)$ notation, which means a quantity that converges to zero as $\lambda$ tends to zero.  We denote by $p'$ the H\"older conjugate of $1 \leq p \leq \infty$ and we
use $a^+$ and $a^-$ to denote $a +\varepsilon$ and $a -\varepsilon$, respectively, for a sufficiently small $\varepsilon > 0$.

\subsection{Linear estimates}

We make use of some mixed space-time Lebesgue spaces, which are defined in order to satisfy the so-called Strichartz estimates. Such estimates are named after an article of Strichartz \cite{Strichartz}, based on estimates for Fourier restrictions on a sphere. They were later generalized by Kato \cite{Kato94}, Keel-Tao \cite{KT98} and by Foschi \cite{Foschi05}.

\begin{de}\label{Hs_adm} If $N \geq 1$ and $s \in (-1,1)$, the pair $(q,r)$ is called $\dot{H}^s$\textit{-admissible} if it satisfies the condition
\begin{equation}\label{hs_adm_eq}
\frac{2}{q} = \frac{N}{2}-\frac{N}{r}-s,    
\end{equation}
where
$$
2 \leq q,r \leq \infty, \text{ and } (q,r,N) \neq (2,\infty,2).
$$
In particular, if $s=0$, we say that the pair is $L^2$-admissible.
\end{de}

\begin{de}\label{As}
Given $N> 2$, consider the set
\begin{equation}
    \mathcal{A}_0 = \left\{(q,r)\text{ is } L^2\text{-admissible} \left|
    \, 2^+ \leq r \leq \frac{2N}{N-2}
    \right.
    \right\}.
\end{equation}
For $N>2$ and $s \in (0,1)$, consider also
\begin{equation}
    \mathcal{A}_s = \left\{(q,r)\text{ is } \dot{H}^{s}\text{-admissible} \left|\,
    \left(\frac{2N}{N-2s}\right)^+ \leq r \leq \left(\frac{2N}{N-2}\right)^-
    \right.
    \right\}
\end{equation}
and
\begin{equation}
    \mathcal{A}_{-s} = \left\{(q,r) \text{ is } \dot{H}^{-s}\text{-admissible} \left|
    \left(\frac{2N}{N-2s}\right)^+ \leq r \leq \left(\frac{2N}{N-2}\right)^-
    \right.\right\}.
\end{equation}

When $N=1,2$ we replace the upper bound for the parameter $r$ in the definitions of $\mathcal{A}_s$ and $\mathcal{A}_{-s}$ by an arbitrarily large number. In the case $s = 1$, which we only consider for $N\geq 3$, we need to employ some values of $r$ which are larger than $2N/(N-2)$. The result which allows us to do so comes from Foschi \cite{Foschi05}, and one needs to be careful in the definition of the corresponding spaces to ensure the validity of the desired estimate. We then let $0<\epsilon_0 < \epsilon_1\ll 1$ be sufficiently small and define

\begin{equation}
    \mathcal{A}_1 = \left\{(q,r)\text{ is } \dot{H}^{1}\text{-admissible} \left|\,
    \frac{2N}{N-2-\epsilon_0/2} \leq r \leq \frac{2N}{N-2-\epsilon_0}
    \right.
    \right\}
\end{equation}
and
\begin{equation}
    \mathcal{A}_{-1} = \left\{(q,r)\text{ is } \dot{H}^{-1}\text{-admissible} \left|\,
    \frac{2N}{N-2+2\epsilon_1} \leq r \leq \frac{2N}{N-2+\epsilon_1}
    \right.
    \right\}.
\end{equation}

Note that, with the above definition of $\mathcal{A}_{\pm1}$, if $(q,r)\in \mathcal{A}_1$ and $(\tilde{q},\tilde{r})\in \mathcal{A}_{-1}$, then $\frac{1}{q}+\frac{1}{\tilde{q}}\in \left[1-\frac{2\epsilon_1-\epsilon_0}{4},1-\frac{\epsilon_1-\epsilon_0}{2}\right] \subsetneq [0,1]$, which enables us to use the  Kato-Strichartz estimates obtained by Foschi, see \cite[Remark 1.11]{Foschi05}.

Given a time-interval $I \subset \mathbb R$ and $0 \leq s\leq 1$, we define the following Strichartz norm
\begin{equation}
	\|u\|_{S(\dot{H}^s,I)} = \sup_{(q,r)\in \mathcal{A}_s}\|u\|_{L_t^q L_x^r(I)},	
\end{equation}
and the dual Strichartz norm
\begin{equation}
	\|u\|_{S'(\dot{H}^{-s},I)} = \inf_{(q,r)\in \mathcal{A}_{-s}}\|u\|_{L_t^{q'}L_x^{r'}(I)}.	
\end{equation}
 If $s=0$, we shall write $S(\dot{H}^0,I) = S(L^2,I)$ and $S'(\dot{H}^0,I) = S'(L^2,I)$. If $I=\mathbb{R}$, we will omit $I$. 

Note that none of these norms allow for the $L^\infty$ norm on time. The main reason for this is to allow the $S(\dot{H}^s)$ norm to be small in small time intervals. Since the $L^\infty_t L^2_x$ norm also plays an important role, we define

\begin{equation}
    \|u\|_{\tilde{S}(L^2,I)} = \|u\|_{{S}(L^2,I)} + \|u\|_{L^\infty_t L^2_x(I)}.
\end{equation}

We also introduce the mixed Sobolev norms
 $$
 \left\|u\right\|_{\dot{S}^s\left(L^2, I\right)}= \left\|D^su\right\|_{S\left(L^2, I\right)},\quad
\left\|u\right\|_{S^s\left(L^2, I\right)}=
 \left\|u\right\|_{S\left(L^2,I\right)}
 +\left\|D^su\right\|_{S\left(L^2, I\right)}
 $$
 and
 $$
 \left\|u\right\|_{\tilde{S}^s\left(L^2, I\right)}=
 \left\|u\right\|_{\tilde{S}\left(L^2,I\right)}
 +\left\|D^su\right\|_{\tilde{S}\left(L^2, I\right)}.
 $$

\end{de}


In this work, we consider the following Kato-Strichartz estimates for $s\geq 0$ (Cazenave \cite{cazenave}, Keel and Tao \cite{KT98}, Kato \cite{Kato94}, Foschi \cite{Foschi05})

\begin{equation}\label{S1}
\|e^{it\Delta}f\|_{
\tilde{S}(L^2)} \lesssim\|f\|_{L^2},
\end{equation}

\begin{equation}\label{S2}
\|e^{it\Delta}f\|_{S(\dot{H}^s)} \lesssim\|f\|_{\dot{H}^s},
\end{equation}
\begin{equation}\label{KS1}
\left\|\int_\mathbb{R}e^{i(t-\tau)\Delta}g(\cdot,\tau) \, d\tau\right\|_{\tilde{S}(L^2,I)} + \left\|\int_0^te^{i(t-\tau)\Delta}g(\cdot,\tau) \, d\tau \right\|_{\tilde{S}(L^2,I)}\lesssim\|g\|_{S'(L^2,I)},
\end{equation}
and
\begin{equation}\label{KS2}
\left\|\int_\mathbb{R}e^{i(t-\tau)\Delta}g(\cdot,\tau) \, d\tau \right\|_{S(\dot{H}^s,I)} + \left\|\int_0^te^{i(t-\tau)\Delta}g(\cdot,\tau) \, d\tau \right\|_{S(\dot{H}^s,I)}\lesssim\|g\|_{S'(\dot{H}^{-s},I)}.
\end{equation}

%
%

\subsection{Fractional chain and Leibniz rules}



We first recall the classical fractional chain rule proved by Christ and Weinstein \cite{CW91} (see also Kenig, Ponce and Vega \cite{KPV93}).

\begin{lem}[Fractional chain rule for Lipschitz-type functions  {\cite{CW91, KPV93}}]\label{chain_2} Let $0<s<1$ and suppose $F \in C(\mathbb C)$ and $G\in C(\mathbb{C},[0,\infty))$ satisfy, for all $u, v$,

\begin{equation}
    |F(u)-F(v)|\lesssim(G(u)+G(v))|u-v|.
\end{equation}
If $1<p,p_1,p_2<+\infty$ are such that $\frac{1}{p}=\frac{1}{p_1}+\frac{1}{p_2}$, then
\begin{equation}
    \|D^s F(u)\|_{L^p} \lesssim \|G(u)\|_{L^{p_1}}\|D^s u\|_{L^{p_2}}.
\end{equation}
\end{lem}

If the desired function is not Lipschitz, but H\"older continuous instead, we have the following version of the chain rule obtained by Visan \cite{Visan07}.

\begin{lem}[Fractional chain rule for H\"older-continuous functions \cite{Visan07}]\label{chain_3} Let $F$ be a H\"older-continuous function of order $0 < \alpha < 1$. Then, for every $0<s<\alpha$, $1<p<\infty$ and $\frac{s}{\alpha}<\sigma<1$, we have

\begin{equation}
    \|D^s F(u)\|_{L^p} \lesssim \|u\|_{L^{\left(\alpha-\frac{s}{\sigma}\right)p_1}}^{\alpha-\frac{s}{\sigma}}\|D^{\sigma} u\|_{L^{\frac{s}{\sigma}p_2}}^{\frac{s}{\sigma}},
\end{equation}
provided $\frac{1}{p}=\frac{1}{p_1}+\frac{1}{p_2}$ and $(1-\frac{s}{\alpha \sigma})p_1>1$.
\end{lem}


We also prove generalizations of the Leibniz rule suitable for singular weights.

\begin{lem}[Generalized Leibniz I]\label{gen_leib_1} Let $0<s<1$, $f, g \in \mathcal{S}(\mathbb{R}^N)$, $A$ be a Lebesgue-measurable set and $\frac{1}{p_i}+\frac{1}{q_i} = \frac{1}{p}$, with $p \in (0,\infty)$, $p_i,q_i\in (1,\infty]$ for $i \in \{1, 2\}$. Then
\begin{align}
    \|D^s(fg)-(D^sf)g - f(D^sg)\|_{L^p}  &\lesssim
    \|f\|_{L^{p_1}(A)}\|D^s g\|_{L^{q_1}}
    +\|f\|_{L^{p_2}(A^c)}\|D^s g\|_{L^{q_2}}
\end{align}

\end{lem}
\begin{proof} Kenig, Ponce and Vega \cite[Theorem A.8]{KPV93} proved
\begin{align}
   D^s(fg)-(D^sf)g - f(D^sg)  
    = T(f,D^s g),
\end{align}
 where $T: L^{p_i}\times L^{q_i} \to L^{p}$, $i \in \{1,2\}$ is a bounded bilinear operator. By writing $f = \mathbbm{1}_{A}f + \mathbbm{1}_{A^c}f$, the result follows.
\end{proof}




The following result follows immediately from Lemma \ref{gen_leib_1} and H\"older inequality.

\begin{cor}[Generalized Leibniz II]\label{gen_leib_2} Let $0<s<1$, $f, g \in \mathcal{S}(\mathbb{R}^N)$, $A, B$ be Lebesgue-measurable sets and $\frac{1}{p_i}+\frac{1}{q_i} = \frac{1}{p}$, with $p \in (1,\infty)$, $q_i,r_i\in (1,\infty]$ for $i \in \{1, 2,3,4\}$. Then
\begin{align}
    \|D^s(fg)\|_{L^p}  &\lesssim
    \|D^sf\|_{L^{p_1}(A)}\|g\|_{L^{q_1}}
    +\|D^sf\|_{L^{p_2}(A^c)}\|g\|_{L^{q_2}}\\
    &\quad+\|f\|_{L^{p_3}(B)}\|D^s g\|_{L^{q_3}}
    +\|f\|_{L^{p_4}(B^c)}\|D^s g\|_{L^{q_4}}
\end{align}
\end{cor}

The above result is essential for dealing with the inhomogeneous term in the nonlinearity, since in view of integrability restrictions we need to consider two different Lebesgue norms in order to study the function $|x|^{-b}$. In the next result, we explore this idea.


\begin{lem}\label{basic_interp} If $a>0$, $ 1< p  < \frac{N}{a}$, $0 \leq s < \frac{N}{p}-a$,  and $q_{\pm\eta}$ is such that $\frac{1}{p} = \frac{a\pm\eta}{N}+\frac{1}{q_{\pm\eta}}$, for small $\eta>0$, then

\begin{equation}
    \| D^s(|x|^{-a} f)\|_{L^p}
    +
    \|D^s(|x|^{-a}f)-(D^s |x|^{-a})f\|_{L^p} 
     \lesssim_{\eta} 
    \left[\|D^s f\|_{L^{q_{\eta}}} \|D^s f\|_{L^{q_{-\eta}}}\right]^\frac{1}{2}
\end{equation}

\begin{proof}
Define $q_{\pm\eta,s}$ as 

\begin{equation}
\frac{1}{p} = \frac{a+s\pm\eta}{N}+\frac{1}{q_{\pm\eta,s}}.    
\end{equation}

For $0<s<1$, using Lemma \ref{gen_leib_2}, H\"older and denoting by $B_R$ the ball of radius $R>0$ centered at the origin, we write
\begin{align}
    \| D^s(|x|^{-a} f)\|_p +
     \|D^s(|x|^{-a}f)-(D^s |x|^{-a})f\|_{p}
     &\\
     &\hspace{-1.5cm}\lesssim\|D^s|x|^{-a}\|_{L^{\frac{N}{a+s+\eta}}(B_{R_1})} \|f\|_{q_{\eta,s}} +
    \|D^s|x|^{-a}\|_{L^{\frac{N}{a+s-\eta}}(B_{R_1}^c)}\|f\|_{q_{-\eta,s}}\\
    &\hspace{-1.5cm}\quad+\||x|^{-a}\|_{L^{\frac{N}{a+\eta}}(B_{R_2})} \|D^s f\|_{q_{\eta,0}}  
    +\||x|^{-a}\|_{L^{\frac{N}{a-\eta}}(B_{R_2}^c)} \|D^s f\|_{q_{\eta,0}} \\
    &\hspace{-1.5
    cm}\lesssim_{\eta} R_1^{\eta} \|f\|_{q_{\eta,s}} + R_1^{-\eta} \|f\|_{q_{-\eta,s}} + R_2^{\eta} \|D^s f\|_{q_{\eta,0}} + R_2^{-\eta} \|D^s f\|_{q_{-\eta,0}}.
\end{align}

Choosing $$R_1^\eta = \left[\frac{\|f\|_{q_{-\eta,s}}}{  \|f\|_{q_{\eta,s}} }\right]^\frac{1}{2}$$
and
$$R_2^\eta = \left[\frac{\|D^sf\|_{q_{-\eta,0}}}{  \|D^s f\|_{q_{\eta,0}} }\right]^\frac{1}{2},$$
we find
$$
 \| D^s(|x|^{-a} f)\|_p
    +
    \|D^s(|x|^{-a}f)-(D^s |x|^{-a})f\|_{p} 
     \lesssim_{\eta} 
    \left[\|f\|_{q_{\eta, s}} \|f\|_{q_{-\eta,s}}\right]^\frac{1}{2}+\left[\|D^s f\|_{q_{\eta,0}} \|D^s f\|_{q_{-\eta,0}}\right]^\frac{1}{2}.
$$

Finally by Sobolev $\|f\|_{q_{\pm\eta,s}}\lesssim \|D^s f\|_{q_{\pm\eta,0}}=\|D^s f\|_{q_{\pm\eta}}$, and we finish the proof for $0<s<1$. The case $s=0$ is analogous, and the case $s \geq 1$ follows from induction, by writing $D^s \sim D^{s-1} \nabla$, iterating the argument and using the classical Leibniz rule $\partial^{\alpha}(fg) = \sum_{\alpha_1+\alpha_2 = \alpha}\partial^{\alpha_1}f \partial^{\alpha_2}g$, for $\alpha, \alpha_1,\alpha_2 \in \mathbb Z^{N}_{\geq 0}$.
\end{proof}
\end{lem}

{To deal with the continuous dependence, we recall a result by Fujiwara \cite[Corollary 1.4]{F22_chain}.

\begin{lem}\label{lem:diff_estim_Hs} Let $\alpha > 0$ and $F \in C^1(\mathbb C)$ be such that $F(0) = F'(0)= 0$ and 
\begin{align}
    |F(u)-F(v)|&\lesssim(|u|+|v|)^{\alpha}|u-v|,\\
    |F'(u) - F'(v)| &\lesssim {\begin{cases}
        (|u|+|v|)^{\alpha-1}|u-v|, &\alpha \geq 1,\\
        |u-v|^{\alpha}, &\alpha <1.
    \end{cases}}
\end{align}
Then, if $0<s<1$, $1\leq p < \infty$, $1<q,r<\infty$, with $1/p=1/q+1/r$, we have
\begin{equation}
    \|D^s(F(u)-F(v))\|_{p} \lesssim \begin{cases}
    (\|u\|_{L^{\alpha q}}+\|v\|_{L^{\alpha  q}})^{\alpha}\|D^s(u-v)\|_{L^r} \\\quad\quad+  (\|D^s u\|_{L^r}+\|D^s v\|_{L^r})( \|u\|_{L^{\alpha q}}
    +\|v\|_{L^{\alpha  q}})^{\alpha-1}\|u-v\|_{L^{\alpha q}}, \quad&\alpha \geq 1\\
    (\|u\|_{L^{\alpha q}}+\|v\|_{L^{\alpha  q}})^{\alpha}\|D^s(u-v)\|_{L^r} 
    \\\quad\quad+  (\|D^s u\|_{L^r}+\|D^s v\|_{L^r})\|u-v\|_{L^{\alpha q}}^\alpha, \quad&\alpha < 1
    \end{cases}
\end{equation}
    
\end{lem}

\begin{re} The conclusion of Lemma \ref{lem:diff_estim_Hs} still follows (with the same proof) if $F$ is only a complex, real-differentiable function, such that $F_z$ and $F_{\bar{z}}$ satisty the hypoteses for $F'$, i.e., if 

\begin{align}
    |F_z(u) - F_z(v)|+|F_{\bar{z}}(u) - F_{\bar{z}}(v)| &\lesssim {\begin{cases}
        (|u|+|v|)^{\alpha-1}|u-v|, &\alpha \geq 1,\\
        |u-v|^{\alpha}, &\alpha <1.
        \end{cases}}
\end{align}
In particular, one can choose $F(z) = |z|^{\alpha} z$.
\end{re}

}

\subsection{Nonlinear estimates}

Now, we use the previous results to deduce suitable nonlinear estimates, which are used in the proof of the local well-posedness results. The main challenge lies in achieving the continuous dependence for small values of $\alpha$ without any loss of regularity (see for instance \cite{CFH11, FH13}). To that end, it is important to estimate the difference of two nonlinear terms at the $H^s$ level. In the case $\alpha\geq 1$ we obtain a Lipschitz bound (see \eqref{dual_grad_s_1}, \eqref{dual_grad_s_3}, \eqref{dual_s_5}) while, for $\alpha< 1$, there is an extra H\"older term which fortunately can be controlled at lower regularity (see \eqref{dual_grad_s_2}, \eqref{dual_grad_s_4}, \eqref{dual_s_6}).
\medskip

Throughout this section, we assume
\begin{equation}\label{hipot}
N \geq 1,\quad s\geq 0,\quad 0 < \alpha <
\begin{cases}
	\frac{4-2b}{N-2s},& s<N/2,\\
	\infty, &s\geq N/2,
\end{cases} \quad \mbox{and } 0 < b <\min\{N/2+1-s,N-s,2\}
\end{equation}
and abbreviate $F(x,u) = |x|^{-b} |u|^ {\alpha}u$ and $F(x,u,v) = |x|^{-b} |u|^ {\alpha}v$.

\begin{lem}
\label{lem_guz} 
Under assumption \eqref{hipot}, if $\alpha$ is an even integer or $s< \alpha+1$, there exists $0<\theta\leq 1$ such that
\begin{align}
\noeqref{dual_l}\label{dual_l}\left\|F(x,u,v)\right\|_{S'\left(L^2, I\right)}&\lesssim|I|^{\frac{\theta\alpha(s-s_c)}{2}}\left\|u\right\|^{\alpha}_{S^s\left(L^2, I\right)}\left\|v\right\|_{S\left(L^2, I\right)},\\
\label{dual_grad_l}\left\|D^sF(x,u)\right\|_{S'\left(L^2, I\right)}
&\lesssim |I|^{\frac{\theta\alpha(s-s_c)}{2}}\left\| u\right\|_{S^s\left(L^2, I\right)}^{{\alpha+1}}.
\end{align}
If $s\le 1$ and $\alpha \geq 1,$
\begin{equation}
\label{dual_grad_s_1}\left\|D^s\left(F(x,u)-F(x,v)\right)\right\|_{S'\left(L^2, I\right)}
\lesssim
\left[|I|^{\frac{\theta(s-s_c)}{2}}\left(\left\|u\right\|_{S^s\left(L^2, I\right)}+\left\|v\right\|_{S^s\left(L^2, I\right)}\right)\right]^{\alpha }\| u-v\|_{S^s(L^2,I)}. 
\end{equation}
If $s\le 1$ and $\alpha < 1,$ there exists $0<\eta<1$ such that
\begin{align}
\label{dual_grad_s_2}\left\|D^s\left(F(x,u)-F(x,v)\right)\right\|_{S'\left(L^2, I\right)}
&\lesssim 
\left[|I|^{\frac{\theta(s-s_c)}{2}}\left(\left\|u\right\|_{S^s\left(L^2, I\right)}+\left\|v\right\|_{S^s\left(L^2, I\right)}\right)\right]^{\alpha }\|u-v\|_{S^s(L^2,I)}\\
&\quad
+\left[|I|^{\frac{\theta(s-s_c)}{2\eta}}\|u-v\|_{S(L^2,I)}\right]^{\eta\alpha }\left(\left\|u\right\|_{S^s\left(L^2, I\right)}+\left\|v\right\|_{S^s\left(L^2, I\right)}\right)^{1+(1-\eta)\alpha}.
\end{align}
\end{lem}

%


\begin{proof}
Throughout the proof, we let $\eta=\eta(N,s,b,\alpha)>0$ be sufficiently small and define precisely $r_i^\pm$ and $q_i^\pm$ as to satisfy $\frac{1}{r_i^\pm}=\frac{1}{r_i}\pm\frac{\eta}{N}$ and $\frac{1}{q_i^\pm}=\frac{1}{q_i}\pm\frac{\eta}{2}$, respectively. We start with the proof of the first two inequalities. To prove \eqref{dual_l}, let $(\gamma,\rho)\in\mathcal{A}_0$ and $\tilde s \leq s$ be such that $\tilde{s} = s$ if $s < N/2$ and  $s_c< \tilde s < N/2$ if $s \geq N/2$. We define the relations
\begin{align}\label{rel1}
    \frac{1}{\rho'} &= \frac{b\mp \eta}{N}+\frac{1}{r_{1}^\pm},\quad \frac{1}{r_{1}} = \frac{\alpha}{r_{2}}+\frac{1}{r_3},\quad \frac{1}{r_{2}}=\frac{1}{r_{3}}-\frac{\tilde s}{N}, \quad \frac{2}{q_3}=\frac{N}{2}-\frac{N}{r_3},
\end{align}
which implies
$$
\frac{1}{\gamma'}=\frac{\alpha(\tilde s - s_c)}{2}+\frac{\alpha+1}{q_3}.
$$
Then use Lemma \ref{basic_interp}, H\"older and Sobolev to write
\begin{align}
    \left\||x|^{-b}|u|^{\alpha}v\right\|_{L^{\gamma'}_tL^{\rho'}_x} &\lesssim \left\|\left[\||u|^{\alpha}v\|_{L^{r_{1}^+}_x}\||u|^{\alpha}v\|_{L^{r_{1}^-}_x}\right]^{\frac{1}{2}}\right\|_{L^{\gamma'}_t}\\
    \label{middle_D_0} &\lesssim 
    |I|^{\frac{\alpha(\tilde{s}-s_c)}{2}}\|u\|^{\alpha}_{L^{q_{3}}_tL^{r_{2}}_x}\left[\|v\|_{L^{q_3^-}_tL^{r_3^+}_x}\|v\|_{L^{q_3^+}_tL^{r_3^-}_x}\right]^{\frac{1}{2}}\\
    &\lesssim  |I|^{\frac{\alpha(\tilde{s}-s_c)}{2}}\|D^{\tilde s}u\|^{\alpha}_{L^{q_{3}}_tL^{r_{3}}_x}\left[\|v\|_{L^{q_3^-}_tL^{r_3^+}_x}\|v\|_{L^{q_3^+}_tL^{r_3^-}_x}\right]^{\frac{1}{2}},
\end{align}
and, invoking also Lemmas \ref{chain_2} and \ref{chain_3} (see also \cite[Proposition 1.3]{F22_chain} for a direct proof in the case $s>1$), 
\begin{align}
    \left\|D^s(|x|^{-b}|u|^{\alpha}u)\right\|_{L^{\gamma'}_tL^{\rho'}_x} &\lesssim \left\|\left[\|D^s(|u|^{\alpha} u)\|_{L^{r_{1}^+}_x}\|D^s(|u|^{\alpha} u)\|_{L^{r_{1}^-}_x}\right]^{\frac{1}{2}}\right\|_{L^{\gamma'}_t}\\
    &\lesssim \label{middle_D_s}
    |I|^{\frac{\alpha(\tilde s-s_c)}{2}}\|u\|_{L^{q_{3}}_tL^{r_{2}}_x}^{\alpha} \left[\|D^s u\|_{L^{q_3^-}_tL^{r_3^+}_x}\|D^s u\|_{L^{q_3^+}_tL^{r_3^-}_x}\right]^{\frac{1}{2}}\\
    &\lesssim 
    |I|^{\frac{\alpha(\tilde s-s_c)}{2}}\|D^{\tilde s}u\|_{L^{q_{3}}_tL^{r_{3}}_x}^{\alpha}\left[\|D^s u\|_{L^{q_3^-}_tL^{r_3^+}_x}\|D^s u\|_{L^{q_3^+}_tL^{r_3^-}_x}\right]^{\frac{1}{2}}.
\end{align}
We then claim it is possible to have $(q_{3}^\pm,r_{3}^\mp),(q_3,r_3)\in \mathcal{A}_0$. Indeed, using the relations \eqref{rel1}, we have
$$
\frac{1}{\rho'}=\frac{b}{N}+\frac{1}{r_1}=\frac{b}{N}+\frac{\alpha+1}{r_3}-\frac{\alpha \tilde{s}}{N}=\frac{b}{N}+\frac{\alpha+1}{r_2}+\frac{\tilde{s}}{N},
$$
with the restrictions $1\leq r_1,r_2<\infty$ and $2<\rho, r_3<2^{\ast}$, where 
\begin{equation}\label{2ast}
2^{\ast}=
\begin{cases}
\infty,& N=1,2\\
\frac{2N}{N-2}, &N\geq 3.
\end{cases}
\end{equation}
Thus, in dimensions $N = 1,2$, such choice is possible if $\rho$ is chosen in the interval
\begin{equation}\label{restriction_rho_1}
    \max\left\{\frac{N}{N-b+\alpha \tilde s},\frac{N}{N-b-\tilde{s}},2\right\}<\rho<\frac{2N}{N-2b - \alpha(N-2 \tilde s)}.
\end{equation}
On the other hand, if $N \geq 3$, we have 
\begin{equation}\label{restriction_rho_2}
    \max\left\{\frac{2N}{N-2b+2+\alpha[2-(N-2\tilde{s})]},\frac{N}{N-b-\tilde{s}},2\right\}<\rho<\min\left\{\frac{2N}{N-2},\frac{2N}{N-2b - \alpha(N-2\tilde{s})}\right\}.
\end{equation}
In both cases, by the definition of $\tilde{s}$, the restriction on $\alpha$ and the fact that $b+s<N/2+1$, simple calculations imply that we can find such $\rho$. 

Now, by defining $0<\theta\leq 1$ as to satisfy
\begin{equation}
    \theta (s-s_c)= \tilde s - s_c,
\end{equation}
the estimates \eqref{dual_l} and \eqref{dual_grad_l} are proved by the embedding $W^{s,r_3} \hookrightarrow \dot{W}^{\tilde s, r_3}$.
{
The proof of \eqref{dual_grad_s_1} and \eqref{dual_grad_s_2} can be done similarly, with the use of Lemma \ref{basic_interp} and \ref{lem:diff_estim_Hs}. In the case $\alpha \geq 1$ and $s\leq 1$, with the same choice of pairs as in \eqref{middle_D_s}, we first apply Lemma \ref{basic_interp} to find
\begin{equation}\label{eq:bound_interm}
\left\|D^s(|x|^{-b}|u|^{\alpha}u- |x|^{-b}|v|^{\alpha}v)\right\|_{L^{\gamma'}_tL^{\rho'}_x} \lesssim \left\|\left[\|D^s(|u|^{\alpha} u-|v|^{\alpha}v)\|_{L^{r_{1}^+}_x}\|D^s(|u|^{\alpha} u-|v|^{\alpha}v)\|_{L^{r_{1}^-}_x}\right]^{\frac{1}{2}}\right\|_{L^{\gamma'}_t}
\end{equation}
Now, using Lemma \ref{lem:diff_estim_Hs} and Hölder inequality, the r.h.s. can be bounded by $A + B + C + D$, where
\begin{align}
	A&= \label{middle_D_s_diff}
	|I|^{\frac{\alpha(\tilde s-s_c)}{2}}\Bigg[\left(\|u\|_{L^{q_{3}}_tL^{r_{2}}_x}+\|v\|_{L^{q_{3}}_tL^{r_{2}}_x}\right)^{\alpha} \left[\|D^s (u-v)\|_{L^{q_3^-}_tL^{r_3^+}_x}\|D^s( u-v)\|_{L^{q_3^+}_tL^{r_3^-}_x}\right]^{\frac{1}{2}} \\ &\lesssim    |I|^{\frac{\alpha(\tilde s-s_c)}{2}}\left(\|D^{\tilde s}u\|_{L^{q_{3}}_tL^{r_{3}}_x}+\|D^{\tilde s}v\|_{L^{q_{3}}_tL^{r_{3}}_x}\right)^{\alpha}\left[\|D^s (u-v)\|_{L^{q_3^-}_tL^{r_3^+}_x}\|D^s (u-v)\|_{L^{q_3^+}_tL^{r_3^-}_x}\right]^{\frac{1}{2}},
\end{align}
\begin{align}
	B&=|I|^{\frac{\alpha(\tilde s-s_c)}{2}}\left(\|u\|_{L^{q_{3}}_tL^{r_{2}}_x}+\|v\|_{L^{q_{3}}_tL^{r_{2}}_x}\right)^{\alpha-1}\left(\|D^s u\|_{L^{q_3}_tL^{r_3}_x}+\|D^s v\|_{L^{q_3}_tL^{r_3}_x}\right) \\&\qquad\times\left[\|u-v\|_{L^{q_3^-}_tL^{r_2^+}_x}\| u-v\|_{L^{q_3^+}_tL^{r_2^-}_x}\right]^{\frac{1}{2}}\\
	&\lesssim |I|^{\frac{\alpha(\tilde s-s_c)}{2}}\left(\|D^{\tilde s}u\|_{L^{q_{3}}_tL^{r_{3}}_x}+\|D^{\tilde s}v\|_{L^{q_{3}}_tL^{r_{3}}_x}\right)^{\alpha-1}\left(\|D^s u\|_{L^{q_3}_tL^{r_3}_x}+\|D^s v\|_{L^{q_3}_tL^{r_3}_x}\right)\\
	&\qquad\times\left[\|D^{\tilde{s}} (u-v)\|_{L^{q_3^-}_tL^{r_3^+}_x}\|D^{\tilde{s}} (u-v)\|_{L^{q_3^+}_tL^{r_3^-}_x}\right]^{\frac{1}{2}},
\end{align}
\begin{align}
	C&= |I|^{\frac{\alpha(\tilde s-s_c)}{2}}\left(\|u\|_{L^{q_{3}}_tL^{r_{2}}_x}+\|v\|_{L^{q_{3}}_tL^{r_{2}}_x}\right)^{\frac{2\alpha-1}{2}}\left(\|D^s u\|_{L^{q_3}_tL^{r_3}_x}+\|D^s v\|_{L^{q_3}_tL^{r_3}_x}\right)^{\frac12} \\&\qquad\times \left[\|D^s (u-v)\|_{L^{q_3^-}_tL^{r_3^+}_x}\| u-v\|_{L^{q_3^+}_tL^{r_2^-}_x}\right]^{\frac{1}{2}}\\
	&\lesssim |I|^{\frac{\alpha(\tilde s-s_c)}{2}}\left(\|D^{\tilde s}u\|_{L^{q_{3}}_tL^{r_{3}}_x}+\| D^{\tilde s}v\|_{L^{q_{3}}_tL^{r_{3}}_x}\right)^{\frac{2\alpha-1}{2}}\left(\|D^s u\|_{L^{q_3}_tL^{r_3}_x}+\|D^s v\|_{L^{q_3}_tL^{r_3}_x}\right)^{\frac12}\\
	&\qquad\times\left[\|D^s (u-v)\|_{L^{q_3^-}_tL^{r_3^+}_x}\| D^{\tilde s}(u-v)\|_{L^{q_3^+}_tL^{r_3^-}_x}\right]^{\frac{1}{2}}
\end{align}
and
\begin{align}
D&= |I|^{\frac{\alpha(\tilde s-s_c)}{2}} |\left(\|u\|_{L^{q_{3}}_tL^{r_{2}}_x}+\|v\|_{L^{q_{3}}_tL^{r_{2}}_x}\right)^{\frac{2\alpha-1}{2}}\left(\|D^s u\|_{L^{q_3}_tL^{r_3}_x}+\|D^s v\|_{L^{q_3}_tL^{r_3}_x}\right)^{\frac12} \\&\qquad\times \left[\|D^s (u-v)\|_{L^{q_3^+}_tL^{r_3^-}_x}\| u-v\|_{L^{q_3^-}_tL^{r_2^+}_x}\right]^{\frac{1}{2}}\\
    &\lesssim 
    |I|^{\frac{\alpha(\tilde s-s_c)}{2}} \left(\|D^{\tilde s}u\|_{L^{q_{3}}_tL^{r_{3}}_x}+\|D^{\tilde s}v\|_{L^{q_{3}}_tL^{r_{3}}_x}\right)^{\frac{2\alpha-1}{2}}\left(\|D^s u\|_{L^{q_3}_tL^{r_3}_x}+\|D^s v\|_{L^{q_3}_tL^{r_3}_x}\right)^{\frac12} \\
    &\qquad\times\left[\|D^s (u-v)\|_{L^{q_3^+}_tL^{r_3^-}_x}\| D^{\tilde s}(u-v)\|_{L^{q_3^-}_tL^{r_2^+}_x}\right]^{\frac{1}{2}},
\end{align}
from which \eqref{dual_grad_s_1} follows. 
In the case $\alpha < 1$, $0 \leq s \leq 1$, we slightly change the choice of $\tilde s$, defining it as $\tilde s = (1 -\eta)s$, with $0<\eta\ll 1$ in the case $s< N/2$. We can bound \eqref{eq:bound_interm} by 
\begin{align}
   \label{middle_D_s_diff_2}
    &|I|^{\frac{\alpha(\tilde s-s_c)}{2}}\Bigg[\left(\|u\|_{L^{q_{3}}_tL^{r_{2}}_x}+\|v\|_{L^{q_{3}}_tL^{r_{2}}_x}\right)^{\alpha} \left[\|D^s (u-v)\|_{L^{q_3^-}_tL^{r_3^+}_x}\|D^s( u-v)\|_{L^{q_3^+}_tL^{r_3^-}_x}\right]^{\frac{1}{2}}\\
    +\ &
    \| u-v\|_{L^{q_{3}}_tL^{r_{2}}_x}^{\alpha}\left[ \left(\|D^s u\|_{L^{q_3^-}_tL^{r_3^+}_x}\|D^s u\|_{L^{q_3^+}_tL^{r_3^-}_x}\right)^{\frac{1}{2}}+\left(\|D^s v\|_{L^{q_3^-}_tL^{r_3^+}_x}\|D^s v\|_{L^{q_3^+}_tL^{r_3^-}_x}\right)^{\frac{1}{2}}\right]\\
    +\ &
\left(\|u\|_{L^{q_{3}}_tL^{r_{2}}_x}+\|v\|_{L^{q_{3}}_tL^{r_{2}}_x}\right)^{\frac{\alpha}{2}} \| u-v\|_{L^{q_{3}}_tL^{r_{2}}_x}^{\frac{\alpha}{2}}\|D^s (u-v)\|_{L^{q_3^-}_tL^{r_3^+}_x}^{\frac{1}{2}}\left( \|D^s u\|_{L^{q_3^-}_tL^{r_3^+}_x}+\|D^s v\|_{L^{q_3^-}_tL^{r_3^+}_x}\right)^{\frac{1}{2}}\\
    +\ &
   \left(\|u\|_{L^{q_{3}}_tL^{r_{2}}_x}+\|v\|_{L^{q_{3}}_tL^{r_{2}}_x}\right)^{\frac{\alpha}{2}} \| u-v\|_{L^{q_{3}}_tL^{r_{2}}_x}^{\frac{\alpha}{2}}\|D^s (u-v)\|_{L^{q_3^+}_tL^{r_3^-}_x}^{\frac{1}{2}}\left( \|D^s u\|_{L^{q_3^+}_tL^{r_3^-}_x}+\|D^s v\|_{L^{q_3^+}_tL^{r_3^-}_x}\right)^{\frac{1}{2}}\Bigg]
    \\\lesssim\  &
    |I|^{\frac{\alpha(\tilde s-s_c)}{2}}\Bigg[\left(\|D^{\tilde s}u\|_{L^{q_{3}}_tL^{r_{3}}_x}+\|D^{\tilde s}v\|_{L^{q_{3}}_tL^{r_{3}}_x}\right)^{\alpha} \left(\|D^s (u-v)\|_{L^{q_3^-}_tL^{r_3^+}_x}+\|D^s( u-v)\|_{L^{q_3^+}_tL^{r_3^-}_x}\right)\\
     &\qquad+
   \|u-v\|_{L^{q_{3}}_tL^{r_{3}}_x}^{\eta\alpha}\|D^s(u-v)\|_{L^{q_{3}}_tL^{r_{3}}_x}^{(1-\eta)\alpha}\\&\qquad\qquad \times\left( \|D^s u\|_{L^{q_3^-}_tL^{r_3^+}_x}+\|D^s u\|_{L^{q_3^+}_tL^{r_3^-}_x}+\|D^s v\|_{L^{q_3^-}_tL^{r_3^+}_x}+\|D^s v\|_{L^{q_3^+}_tL^{r_3^-}_x}\right)\Bigg]
   \\\lesssim\  &
    |I|^{\frac{\alpha(\tilde s-s_c)}{2}}\Bigg[\left(\|D^{\tilde s}u\|_{L^{q_{3}}_tL^{r_{3}}_x}+\|D^{\tilde s}v\|_{L^{q_{3}}_tL^{r_{3}}_x}\right)^{\alpha} \left(\|D^s (u-v)\|_{L^{q_3^-}_tL^{r_3^+}_x}+\|D^s( u-v)\|_{L^{q_3^+}_tL^{r_3^-}_x}\right)\\
     & \qquad+
   \|u-v\|_{L^{q_{3}}_tL^{r_{3}}_x}^{\eta\alpha}\left(\|D^s u\|_{L^{q_{3}}_tL^{r_{3}}_x}^{(1-\eta)\alpha}+\|D^s u\|_{L^{q_{3}}_tL^{r_{3}}_x}^{(1-\eta)\alpha}\right)\\ & \qquad \qquad\times\left( \|D^s u\|_{L^{q_3^-}_tL^{r_3^+}_x}+\|D^s u\|_{L^{q_3^+}_tL^{r_3^-}_x}+\|D^s v\|_{L^{q_3^-}_tL^{r_3^+}_x}+\|D^s v\|_{L^{q_3^+}_tL^{r_3^-}_x}\right)\Bigg],
    \end{align}
from which estimate \eqref{dual_grad_s_2} follows.



}

\end{proof}

\begin{lem}
	\label{lem_guz2} 
	Under assumption \eqref{hipot}, if $0 \leq s_c< s \leq 1$, $s< \frac{N}{2}$, then
	\begin{align}
		\label{dual_grad_l_3}\left\|D^sF(x,u)\right\|_{S'\left(L^2, I\right)}
		&\lesssim \left[|I|^{\frac{s-s_c}{2}}\left\|u\right\|_{\dot{S}^s\left(L^2, I\right)}\right]^{\alpha }\| u\|_{\dot{S}^s(L^2,I)}.
	\end{align}
	Assuming that $\alpha\geq 1$ when $N=1$, there exists $0<{s_0}\leq\min\{\alpha,s\}$ such that 
	\begin{align}\label{dual_grad_l_2}
		\left\|D^{s_0}F(x,u,v)\right\|_{S'\left(\dot{H}^{-(s-s_0)}, I\right)}
		\lesssim&\left[|I|^{\frac{s-s_c}{2}}\left\|u\right\|_{\dot{S}^s\left(L^2, I\right)}\right]^{\alpha }\|D^{s_0} v\|_{S(\dot{H}^{s-s_0},I)}.
	\end{align}
	Furthermore, in the case  $\alpha \geq 1,$
	\begin{equation}
		\label{dual_grad_s_3}\left\|D^s\left(F(x,u)-F(x,v)\right)\right\|_{S'\left(L^2, I\right)}
		\lesssim \left[|I|^{\frac{s-s_c}{2}}\left(\left\|u\right\|_{\dot{S}^s\left(L^2, I\right)}+\left\|v\right\|_{\dot{S}^s\left(L^2, I\right)}\right)\right]^{\alpha }\|u-v\|_{\dot{S}^ s(L^2,I)}, 
	\end{equation}
	and in the case $\alpha < 1$,
	\begin{align}
		\label{dual_grad_s_4}\left\|D^s\left(F(x,u)-F(x,v)\right)\right\|_{S'\left(L^2, I\right)}
		&\lesssim 
		\left[|I|^{\frac{s-s_c}{2}}\left(\left\|u\right\|_{\dot{S}^s\left(L^2, I\right)}+\left\|v\right\|_{\dot{S}^s\left(L^2, I\right)}\right)\right]^{\alpha }\|u-v\|_{\dot{S}^s(L^2,I)}\\
		&\quad+
		\left[|I|^{\frac{s-s_c}{2}}\|D^{s_0}(u-v)\|_{S(\dot{H}^{s-s_0},I)}\right]^{\alpha }\left(\left\|u\right\|_{\dot{S}^s\left(L^2, I\right)}+\left\|v\right\|_{\dot{S}^s\left(L^2, I\right)}\right).
	\end{align}
\end{lem}
\begin{proof}
	 In the previous proof, since  $\tilde s = s$ if $s<N/2$, estimate \eqref{dual_grad_l_3} follows from the same computations as for \eqref{dual_grad_l}. Similarly, the proofs of \eqref{dual_grad_s_3} and \eqref{dual_grad_s_4} are completely analogous to those of \eqref{dual_grad_s_1} and \eqref{dual_grad_s_2}.
	 
	 It remains to prove \eqref{dual_grad_l_2}. We first consider the case $\alpha<1$ and $N\geq 2$. Define $\sigma=(1-\eta)s $ and $s_0 = (1-\eta)^2\alpha s$, so that $0<s_0<\alpha$ and $\frac{s_0}{\alpha}<\sigma<s\leq 1$. Let $(\gamma, \rho) \in \mathcal{A}_{-(s-s_0)}$, $(q_5^\mp,r_5^\pm)\in \mathcal{A}_{s-s_0}$, and $(q_8,r_8)\in \mathcal{A}_{0}$. From Lemmas \ref{basic_interp},  \ref{gen_leib_1} and \ref{chain_3} and Sobolev embedding, we have
	 \begin{align}\label{inequalities_s_c}
	 	\|D^{s_0}(|x|^{-b}|u|^\alpha v)\|_{ L^{\rho'}_x} &\lesssim \left[\|D^{s_0}(|u|^{\alpha}v)\|_{L^{r_1^+}_x}\|D^{s_0}(|u|^{\alpha}v)\|_{L^{r_1^-}_x}\right]^\frac{1}{2}\\
	 	&\lesssim \Bigg[\left(\|D^{s_0}(|u|^{\alpha})\|_{L^{r_2}_x}\|v\|_{L^{r_3^+}_x}+\||u|^{\alpha}\|_{L^{r_4}_x}\|D^{s_0}v\|_{L^{r_5^+}_x}\right)\times\\
	 	&\hspace{2cm}\times 
	 	\left(\|D^{s_0}(|u|^{\alpha})\|_{L^{r_2}_x}\|v\|_{L^{r_3^-}_x}+\||u|^{\alpha}\|_{L^{r_4}_x}\|D^{s_0}v\|_{L^{r_5^-}_x}\right)\Bigg]^\frac{1}{2}\\
	 	&\lesssim \left(\|u\|^{\alpha-\frac{s_0}{\sigma}}_{L^{\left(\alpha-\frac{s_0}{\sigma}\right)r_6}_x}\|D^\sigma u\|^{\frac{s_0}{\sigma}}_{L^{\frac{s_0}{\sigma} r_7}_x}+\|D^s u\|_{L^{r_8}_x}^{\alpha}\right)\left[\|D^{s_0}v\|_{L^{r_5^+}_x}\|D^{s_0}v\|_{L^{r_5^-}_x}\right]^{\frac{1}{2}}\\
	 	&\lesssim \|D^s u\|_{L^{r_8}_x}^{\alpha}\left[\|D^{s_0}v\|_{L^{r_5^+}_x}\|D^{s_0}v\|_{L^{r_5^-}_x}\right]^{\frac{1}{2}},
	 \end{align}
	 where
	 \begin{align}
	 	\frac{1}{\rho'} = \frac{b\mp \eta}{N} + \frac{1}{r_{1}^\pm},\quad 
	 	\frac{1}{r_1} &= \frac{1}{r_2}+\frac{1}{r_3}=\frac{1}{r_4}+\frac{1}{r_5},\quad
	 	\frac{1}{r_2}=\frac{1}{r_6}+\frac{1}{r_7}, \quad \left(\alpha-\frac{s_0}{\sigma}\right)r_6>1,\\
	 	\frac{1}{r_3} = \frac{1}{r_5}-\frac{s_0}{N}, \quad \frac{1}{r_8}&=\frac{1}{\alpha r_4}+\frac{s}{N}=\frac{1}{(\alpha-\frac{s_0}{\sigma})r_6}+\frac{s}{N}=\frac{1}{\frac{s_0}{\sigma}r_7}+\frac{s-\sigma}{N}.\label{reln=1}
	 \end{align}
	 We choose $\frac{1}{r_8}=\frac{1}{2}-\eta$, so that $0<1/\left[\left(\alpha-\frac{s_0}{\sigma}\right)r_6\right]=1/2-s/N-\eta<1$, and $r_5=\rho$.
	 With this choice and recalling that $\alpha = \frac{4-2b}{N-2s_c}$, the above relations imply
	 $$
	 \frac{1}{\rho}=\frac{N-2}{2N}+\frac{\alpha(s-s_c)}{2N}+\frac{\alpha \eta}{2}\quad \mbox{and} \quad \frac{1}{\gamma'}=\frac{\alpha(s-s_c)}{2}+\frac{\alpha }{q_8}+\frac{1}{q_5}.
	 $$
	 In order to have $(\gamma, \rho) \in \mathcal{A}_{-(s-s_0)}$, we need to satisfy the condition
	 \begin{equation}
	 	\frac{N-2}{2N}<\frac{1}{\rho}<\frac{N-2(s-s_0)}{2N},
	 \end{equation}
	 which is true, for $\eta>0$ sufficiently small, since $\alpha(1+s_c)>0$. By then calculating the $L^{\gamma'}_t$-norm and using H\"older and the corresponding scaling relations, one obtains \eqref{dual_grad_l_2} in this case. 
	 
	 The case $\alpha \geq 1$ follows by setting $s_0 = s$ and using the same spaces as in \eqref{middle_D_s} to write
	 \begin{equation}
	 	\left\|D^s(|x|^{-b}|u|^{\alpha}v)\right\|_{L^{\gamma'}_tL^{\rho'}_x}
	 	\lesssim 
	 	|I|^{\frac{\alpha( s-s_c)}{2}}\|D^{ s}u\|_{L^{q_{3}}_tL^{r_{3}}_x}\left[\|D^s u\|_{L^{q_3^-}_tL^{r_3^+}_x}\|D^s v\|_{L^{q_3^+}_tL^{r_3^-}_x}\right]^{\frac{1}{2}}.
	 \end{equation}
\end{proof}

\begin{lem}
	\label{lem_guz3} 
	Under assumption \eqref{hipot}, if $0\leq s_c\leq1$:
	\begin{align}
		\label{dual_s1}\left\|F(x,u,v)\right\|_{S'(\dot{H}^{-s_c},I)}&\lesssim\left\|u\right\|^{\alpha}_{S(\dot{H}^{s_c},I)}\left\|v\right\|_{S(\dot{H}^{s_c},I)},\\
		\label{dual_s2}\left\|D^{s_c}F(x,u)\right\|_{S'(L^2,I)}&\lesssim\left\|u\right\|^{\alpha}_{S(\dot{H}^{s_c},I)}\left\|u\right\|_{\dot{S}^{s_c}(L^2,I)}.
	\end{align}
	Moreover, if $\alpha \geq 1$,
	\begin{align}
		\label{dual_s_5}\left\|D^{s_c}(F(x,u)-F(x,v))\right\|_{S'(L^2,I)}&\lesssim\left(\left\|u\right\|_{S(\dot{H}^{s_c},I)}+\left\|v\right\|_{S(\dot{H}^{s_c},I)}\right)^{\alpha}\left\|u-v\right\|_{\dot{S}^{s_c}(L^2,I)} \\
		&\hspace{-3cm}+ \left(\left\|u\right\|_{S(\dot{H}^{s_c},I)}+\left\|v\right\|_{S(\dot{H}^{s_c},I)}\right)^{\alpha-1} \left(\left\| u\right\|_{\dot{S}^{s_c}(L^2,I)}+\left\| u\right\|_{\dot{S}^{s_c}(L^2,I)}\right)\left\|u-v\right\|_{\dot{S}^{s_c}(L^2,I)},
	\end{align}
	and, if $\alpha<1$,
	\begin{align}
		\label{dual_s_6}\left\|D^{s_c}(F(x,u)-F(x,v))\right\|_{S'(L^2,I)}&\lesssim\left(\left\|u\right\|_{S(\dot{H}^{s_c},I)}+\left\|v\right\|_{S(\dot{H}^{s_c},I)}\right)^{\alpha}\left\|u-v\right\|_{\dot{S}^{s_c}(L^2,I)}\\
		& \quad+ \left\|u-v\right\|^{\alpha}_{S(\dot{H}^{s_c},I)}\left(\left\|u\right\|_{\dot{S}^{s_c}(L^2,I)}+\left\|v\right\|_{\dot{S}^{s_c}(L^2,I)}\right).
	\end{align}
\end{lem}
\begin{proof}
	We first prove \eqref{dual_s1}. Let $(\gamma, \rho) \in \mathcal{A}_{-s_c}$ and write
	\begin{align}
		\left\||x|^{-b}|u|^{\alpha}v\right\|_{L^{\gamma'}_tL^{\rho'}_x} &\lesssim \left\|\left[\||u|^{\alpha}v\|_{L^{r_{1}^+}_x}\||u|^{\alpha}v\|_{L^{r_{1}^-}_x}\right]^{\frac{1}{2}}\right\|_{L^{\gamma'}_t}\\
		&\lesssim 
		\|u\|^{\alpha}_{L^{q_{2}}_tL^{r_{2}}_x}\left[\|v\|_{L^{q_3^-}_tL^{r_3^+}_x}\|v\|_{L^{q_3^+}_tL^{r_3^-}_x}\right]^{\frac{1}{2}},
	\end{align}
	where
	
	\begin{equation}
		\frac{1}{\rho'} = \frac{b\pm\eta}{N}+\frac{1}{r_1^\pm}, \quad \frac{1}{r_1}=\frac{\alpha}{r_2} + \frac{1}{r_3}, \quad (q_j, r_j) \in \mathcal{A}_{s_c}, \,\, j=2,3 .
	\end{equation}
	We start considering $s_c<1$. If $N\geq 2$, we set $\frac{1}{r_2}=\frac{N-2s_c}{2N}-\eta$ and $r_3 = \rho$ to get
	$$
	\frac{1}{\rho}=\frac{N-2}{2N}+\frac{\alpha\eta}{2}.
	$$
	On the other hand, if $N=1$, we set $\frac{1}{\rho}=\eta$ and $r_3 = \rho$ to deduce
	$$
	\frac{1}{r_2}=\frac{1-b-2\eta}{\alpha}<\frac{2-b}{\alpha}=\frac{1-2s_c}{2}.
	$$
	In both cases, we satisfy the required conditions to ensure $(\gamma, \rho) \in \mathcal{A}_{-s_c}$ and $(q_j, r_j) \in \mathcal{A}_{s_c}, \,\, j=2,3$.
	It remains to consider the case $s_c = 1$, which requires $N \geq 3$. We now choose fix $0<\epsilon\ll1$, impose $0<\eta\ll\epsilon$ and set $r_2 = r_3= \frac{2N}{N-2-\epsilon}$ to obtain 
	\begin{equation}
		\frac{1}{\rho} = \frac{N-2+(1+(4-2b)/(N-2))\epsilon}{2N}.
	\end{equation}
	Therefore, choosing $\epsilon_0, \epsilon_1$ in the definition of $\mathcal{A}_{\pm1}$ (c.f. Definition \ref{As}) as 
	
	\begin{equation}
		\epsilon_0 = \left[1+\frac{2-b}{2(N-2)}\right]\epsilon, \quad \epsilon_1 =\left(1+\frac{2-b}{N-2}\right)\epsilon,
	\end{equation}
	we ensure $(\gamma, \rho)\in \mathcal{A}_{-1}$ and $(q_2,r_2),(q_3^{\mp},r_3^{\pm}) \in\mathcal{A}_{1}$. 
	
	It remains to prove \eqref{dual_s2}. In the same fashion as \eqref{middle_D_s}, for $(\gamma,\rho)\in\mathcal{A}_0$, we write
	\begin{equation}
		\left\|D^{s_c}(|x|^{-b}|u|^{\alpha}u)\right\|_{L^{\gamma'}_tL^{\rho'}_x}\lesssim 
		\|u\|_{L^{q_{2}}_tL^{r_{2}}_x}^{\alpha} \left[\|D^{s_c} u\|_{L^{q_3^-}_tL^{r_3^+}_x}\|D^{s_c} u\|_{L^{q_3^+}_tL^{r_3^-}_x}\right]^{\frac{1}{2}},
	\end{equation}
	with the relations
	
	\begin{equation}
		\frac{1}{\rho'} = \frac{b\mp \eta}{N} + \frac{1}{r_{1}^\pm},\quad 
		\frac{1}{r_1} = \frac{\alpha}{r_2}+\frac{1}{r_3}, \quad (q_2,r_2) \in \mathcal{A}_{s_c}, \quad (q_3^\mp,r_3^\pm) \in \mathcal{A}_0.
	\end{equation}
	As in the proof of \eqref{dual_s1}, when $s_c<1$ we choose $\frac{1}{r_2}=\frac{N-2s_c}{2N}-\eta$ and $r_3 = \rho$ if $N\geq 2$, and $\frac{1}{\rho}=\eta$ and $r_3 = \rho$ if $N=1$. For the case $s_c=1$, which requires $N\geq 3$, we set $r_2 = \frac{2N}{N-2-\epsilon}$, $r_3 = \rho$ and obtain
	$$
	\frac{1}{\rho}=\frac{N-2+\epsilon(2-b)/(N-2)}{2N}>\frac{N-2}{2N}.
	$$ 
	For all these choices we ensure $(\gamma, \rho), (q_3^\mp,r_3^\pm) \in \mathcal{A}_0$ and $(q_2,r_2) \in \mathcal{A}_1$.
	{
		Finally, the same choice of indices give, in the case $\alpha \geq 1$,
		
		\begin{align}
			&\left\|D^{s_c}(|x|^{-b}|u|^{\alpha}u-|x|^{-b}|v|^ \alpha v)\right\|_{L^{\gamma'}_tL^{\rho'}_x}\\\lesssim\ & \left(\|u\|_{L^{q_{2}}_tL^{r_{2}}_x}+\|v\|_{L^{q_{2}}_tL^{r_{2}}_x}\right)^{\alpha} \left[\|D^{s_c} (u-v)\|_{L^{q_3^-}_tL^{r_3^+}_x}\|D^{s_c} (u-v)\|_{L^{q_3^+}_tL^{r_3^-}_x}\right]^{\frac{1}{2}}\\
			+\ &\left(\|u\|_{L^{q_{2}}_tL^{r_{2}}_x}+\|v\|_{L^{q_{2}}_tL^{r_{2}}_x}\right)^{\alpha-1}\left(\|D^{s_c}u\|_{L^{q_{3}}_tL^{r_{3}}_x}+\|D^{s_c}v\|_{L^{q_{3}}_tL^{r_{3}}_x}\right) \left[\|u-v\|_{L^{q_2^-}_tL^{r_2^+}_x}\|u-v\|_{L^{q_2^+}_tL^{r_2^-}_x}\right]^{\frac{1}{2}},
		\end{align}
		and, in the case $\alpha < 1$,
		\begin{align}
			&\left\|D^{s_c}(|x|^{-b}|u|^{\alpha}u-|x|^{-b}|v|^ \alpha v)\right\|_{L^{\gamma'}_tL^{\rho'}_x}\\\lesssim\ & \left(\|u\|_{L^{q_{2}}_tL^{r_{2}}_x}+\|v\|_{L^{q_{2}}_tL^{r_{2}}_x}\right)^{\alpha} \left[\|D^{s_c} (u-v)\|_{L^{q_3^-}_tL^{r_3^+}_x}\|D^{s_c} (u-v)\|_{L^{q_3^+}_tL^{r_3^-}_x}\right]^{\frac{1}{2}}\\
		+\ & \| u-v\|_{L^{q_{2}}_tL^{r_{2}}_x}^{\alpha}\left[ \left(\|D^{s_c} u\|_{L^{q_3^-}_tL^{r_3^+}_x}\|D^{s_c} u\|_{L^{q_3^+}_tL^{r_3^-}_x}\right)^{\frac{1}{2}}+\left(\|D^{s_c} v\|_{L^{q_3^-}_tL^{r_3^+}_x}\|D^{s_c} v\|_{L^{q_3^+}_tL^{r_3^-}_x}\right)^{\frac{1}{2}}\right],
		\end{align}
		concluding the proof of inequalities \eqref{dual_s_5} and \eqref{dual_s_6}.

	}
\end{proof}


\section{Well-posedness}\label{Well}
In this section, we show the existence results stated in Theorems \ref{InhTheo} and \ref{HomTheo}.
\subsection{Well-posedness in inhomogeneous Sobolev spaces}

\begin{proof}[Proof of Theorem \ref{InhTheo} (a)] Fixed $u_0 \in {H}^{s}$, let 
\begin{align}
    E = \Big\{u \in L^{\infty}_t{H}^s_x([0,T])\cap S^s(L^2,[0,T]) :\,
\| u\|_{\tilde{S}^s(L^2,[0,T])} \leq 2C \|u_0\|_{{H}^{s}} \Big\}
\end{align}
be a (complete) metric space with the metric
$$
\rho(u,v) = \|u-v\|_{S(L^2,[0,T])}.
$$
Then, applying the linear estimates \eqref{S2} and \eqref{KS2}, and the nonlinear estimates \eqref{dual_l}, \eqref{dual_grad_l} and \eqref{dual_grad_s_1} to the operator \eqref{eq:duhamelop}, we obtain
\begin{align}
    \| \Phi(u)\|_{\tilde{S}^s(L^2,[0,T])} 
    &\leq C \|u_0\|_{{H}^{s}} + CT^{\frac{\alpha(s-s_c)}{2}}\|u\|_{ S^s(L^2,[0,T])}^{\alpha+1}\\
    &\leq C[1+O(T^\frac{s-s_c}{2}\|u_0\|_{\dot{H}^s})^{\alpha}]\|u_0\|_{{H}^{s}}
    \end{align}
    and
  \begin{align}  
    \|\Phi(u)-\Phi(v)\|_{S(L^2,[0,T])} 
    &\leq O(T^\frac{s-s_c}{2}\|u_0\|_{{H}^s})^{\alpha}\|u-v\|_{S(L^2,[0,T])}.
\end{align}
Thus, choosing $T = \delta\|u_0\|_{H^s}^{-\frac{2}{s-s_c}}$ for $\delta >0$ small enough (depending only on universal constants) we deduce the result by standard arguments.

{


We now deal with the continuous dependence on the initial data in the case $0\leq s \leq 1$. The argument is split in the cases $\alpha\geq 1$ and $0<\alpha < 1$. For the former case, the previous estimates and the Duhamel formula give, for $T>0$ such that two solutions $u$ and $v$ are defined on $[0,T]$,

\begin{align}
    \|u - v\|_{\tilde{S}^s(L^2,[0,T])} &\lesssim \|u_0 - v_0\|_{H^s} \\&\quad+T^{\theta \alpha(s-s_c)/2}\left(\|u\|_{S^s(L^2,[0,T])} + \|v\|_{S^s(L^2,[0,T])}\right)^{\alpha}
    \|u-v\|_{S^s(L^2,[0,T])}
    \\&\lesssim \|u_0 - v_0\|_{H^s} + T^{\theta \alpha(s-s_c)/2}\left(\|u_0\|_{H^s} + \|v_0\|_{H^s}\right)^{\alpha}
\|u-v\|_{S^s(L^2,[0,T])}.
\end{align}
Now, if $\|u_0 - v_0\|_{H^s}  \ll {\|u_0\|_{H^s}}$, our choice of $T$ allows us to conclude (by possibly decreasing $\delta$)
\begin{equation}
    \|u - v\|_{\tilde{S}^s(L^2,[0,T])} \lesssim \|u_0 - v_0\|_{H^s}.
\end{equation}
For the case $0<\alpha<1$, we first establish control on lower-regularity norms. By \eqref{dual_l},
\begin{align}
     \|u - v\|_{\tilde{S}(L^2,[0,T])}
    &\lesssim \|u_0 - v_0\|_{L^2} + T^{ \theta\alpha(s-s_c)/2}\left(\|u_0\|_{H^s} + \|v_0\|_{H^s}\right)^{\alpha}
\|u-v\|_{S(L^2,[0,T])},
\end{align}
which, if $\|u_0 - v_0\|_{H^s} \ll \|u_0\|_{H^s}$, allows us to conclude
\begin{align}
     \|u - v\|_{\tilde{S}(L^2,[0,T])}
    &\lesssim \|u_0 - v_0\|_{L^2}.
\end{align}
The control on the lower-regularity norm can thus be passed to the higher-regularity norm by using \eqref{dual_grad_s_2}:
\begin{align}
     \|D^s(u - v)\|_{\tilde S (L^2,[0,T])}
    &\lesssim \|u_0 - v_0\|_{H^s} + T^{ \alpha(s-s_c)/2}\left(\|u_0\|_{H^s} + \|v_0\|_{H^s}\right)^{\alpha}
\|u-v\|_{\dot{S}^s(L^2,[0,T])}\\
&\quad+T^{ \alpha(s-s_c)/2}\|u-v\|_{S(L^2,[0,T])}^{\eta \alpha}\left(\|u_0\|_{H^s} + \|v_0\|_{H^s}\right)^{1+(1-\eta)\alpha},
\end{align}
which, by the previous estimates, implies
\begin{align}
     \|D^s(u - v)\|_{\tilde{S}(L^2,[0,T])}
    &\lesssim \|u_0 - v_0\|_{H^s} + \|u_0 - v_0\|_{H^s}^{\eta\alpha}.
\end{align}

}
\end{proof}

\begin{proof}[Proof of Theorem \ref{InhTheo} (b)]
 In this case, for $\delta_0>0$ to be chosen later, we let $T_0>0$ be such that
\begin{equation}\label{delta_0_bound}
    \|e^{it\Delta}u_0\|_{S^s(L^2,[0,T_0])}<\delta_0.
\end{equation}
We then define the complete metric space $(E,\rho)$ by

\begin{align}
E = \Big\{u \in \tilde{S}^s(L^2,[0,T]) :\,
& \|u\|_{L^\infty_tH^s_x([0,T_0])}\leq 2C \|u_0\|_{H^s},\\
&\| u\|_{S^s(L^2,[0,T_0])} \leq 2 \|e^{it\Delta}u_0\|_{S^s(L^2,[0,T_0])}\Big\},
\end{align}
\begin{equation}
    \rho(u,v) = \|u-v\|_{S(L^2,[0,T_0])}.
\end{equation}
Thus, using the same ideas as in part (a) and \eqref{delta_0_bound}, we have
\begin{align}
    \|\Phi(u)\|_{L^\infty_tH^s_x([0,T_0])}&\leq C\|u_0\|_{H^s}+C\|u\|_{S^s(L^2,[0,T_0])}^{\alpha+1}\\
    &\leq [1+O(\delta_0)^\alpha]C\|u_0\|_{H^s}\\
    \| \Phi(u)\|_{S^s(L^2,[0,T_0])} 
    &\leq \|e^{it\Delta}u_0\|_{S^s(L^2,[0,T_0])}  + C\|u\|_{S^s(L^2,[0,T_0])}^{\alpha+1}\\
    &\leq [1+O(\delta_0)^{\alpha}]\|e^{it\Delta}u_0\|_{S^s(L^2,[0,T_0])}
    \end{align}
    and
    \begin{align}
    \|\Phi(u)-\Phi(v)\|_{S(L^2,[0,T_0])} 
    &\leq o_{\delta_0}(1)\|u-v\|_{S(L^2,[0,T_0])}.
\end{align}
The result follows by choosing $\delta_0 >0$ small enough, depending only on universal constants. Note that, if 

\begin{equation}
    \|e^{it\Delta}u_0\|_{S^s(L^2,\mathbb R)}<\delta_0,
\end{equation}
then the solution exists for all $t \in \mathbb{R}$.

{For the continuous dependence in this case, if $u_0 \in H^s$, we choose again $T_0$ as in \eqref{delta_0_bound} and note that, if $v_0 \in H^s$ is such that $\|u_0-v_0\|_{H^s} \ll \delta_0$, then the corresponding solutions $u$ and $v$ are defined on $[0,T_0]$ satisfing
\begin{align}
    \|u\|_{L^{\infty}_{[0,T_0]}H^s_x} + \|v\|_{L^{\infty}_{[0,T_0]}H^s_x} &\lesssim \|u_0\|_{H^s}
\end{align}
and
\begin{align}
    \|u\|_{S^s(L^2,[0,T_0])}+\|v\|_{S^s(L^2,[0,T_0])} &\lesssim \delta_0.\\
\end{align}
By Duhamel and Strichartz, we also have
\begin{align}
    \|u-v\|_{\tilde{S}(L^2,[0,T_0])} &\lesssim \|u_0-v_0\|_{L^2} + \left(\|u\|_{S^s(L^2,[0,T_0])}+\|v\|_{S^s(L^2,[0,T_0])}\right)^\alpha\|u-v\|_{S(L^2,[0,T_0])}\\
    &\lesssim \|u_0-v_0\|_{L^2} + \delta_0^\alpha\|u-v\|_{S(L^2,[0,T_0])},
\end{align}
which implies, by possibly choosing a smaller (but universal) $\delta_0$, 
\begin{equation}
    \|u-v\|_{\tilde{S}(L^2,[0,T_0])} \lesssim \|u_0-v_0\|_{L^2} .
\end{equation}
Moreover, we have, for $\alpha \geq 1$
\begin{align}
    \|u-v\|_{\tilde{S}^s(L^2,[0,T_0])} &\lesssim\|u_0-v_0\|_{H^s} + \left(\|u\|_{S^s(L^2,[0,T_0])}+\|v\|_{S^s(L^2,[0,T_0])}\right)^\alpha\|u-v\|_{S^s(L^2,[0,T_0])}\\
    &\lesssim \|u_0-v_0\|_{L^2} + \delta_0^\alpha\|u-v\|_{S^s(L^2,[0,T_0])},
\end{align}
which implies $\|u-v\|_{\tilde{S}^s(L^2,[0,T_0])}\lesssim\|u_0-v_0\|_{H^s} $.
In the case $\alpha < 1$, the estimate is
\begin{align}
    \|D^{s_0}(u-v)\|_{S(\dot{H}^{s-s_0},[0,T_0])} &\lesssim
    \|u_0-v_0\|_{H^s} + \left(\|u\|_{S^s(L^2,[0,T_0])}+\|v\|_{S^s(L^2,[0,T_0])}\right)^\alpha\\
    &\quad \times \|D^{s_0}(u-v)\|_{S(\dot{H}^{s-s_0},[0,T_0])}
\end{align}
which implies $\|D^{s_0}(u-v)\|_{S(\dot{H}^{s-s_0},[0,T_0])} \lesssim\|u_0-v_0\|_{H^s}$. This control can thus be used to bound
\begin{align}
    \|u-v\|_{\tilde{S}^s(L^2,[0,T_0])} &\lesssim\|u_0-v_0\|_{H^s} + \left(\|u\|_{S^s(L^2,[0,T_0])}+\|v\|_{S^s(L^2,[0,T_0])}\right)^\alpha\|u-v\|_{S^s(L^2,[0,T_0])}\\
    &\quad+ \|D^{s_0}(u-v)\|_{S(\dot{H}^{s-s_0},[0,T_0])}^{\alpha }\left(\left\|u\right\|_{S^s\left(L^2, [0,T_0]\right)}+\left\|v\right\|_{S^s\left(L^2, [0,T_0]\right)}\right)\\
    &\lesssim \|u_0-v_0\|_{H^s} + \delta_0^ {\alpha}\|u-v\|_{S^s(L^2,[0,T_0])}+\delta_0\|u_0-v_0\|_{H^s}^{\alpha},
\end{align}
which is enough to conclude the claimed continuous dependence.
}

%


\end{proof}
\subsection{Well-posedness in homogeneous Sobolev spaces}

\begin{proof}[Proof of Theorem \ref{HomTheo} (a)]
Fixed $u_0 \in \dot{H}^{s} $, let 

\begin{align}
E = \Big\{ \tilde {S}^s(L^2,[0,T]) 
\cap L^{\infty}_t L^{\frac{2N}{N-2s}}_x([0,T]\times \mathbb R^N): \,
\|D^s u\|_{\tilde{S}(L^2,[0,T])} \leq 2C \|u_0\|_{\dot{H}^{s}},
\Big\}    
\end{align}
 be a metric space with the metric

$$
\rho(u,v) = \|D^{s_0}(u-v)\|_{S(\dot{H}^{s-s_0},[0,T])},
$$
where $s_0$ is given in Lemma \ref{lem_guz2}.


Thus, from inequalities \eqref{dual_grad_l_2} and \eqref{dual_grad_l_3}, we deduce
\begin{align}
    \|D^{s} \Phi(u)\|_{\tilde{S}(L^2,[0,T])} 
    &\leq C \|u_0\|_{\dot{H}^{s}} + CT^{\frac{\alpha(s-s_c)}{2}}\|D^{s}u\|_{S(L^2,[0,T])}^{\alpha+1}\\
    &\leq C[1+O(T^\frac{s-s_c}{2}\|u_0\|_{\dot{H}^s})^{\alpha}]\|u_0\|_{\dot{H}^{s}}
    \end{align}
    and
\begin{align}    
    \|D^{s_0}[\Phi(u)-\Phi(v)]\|_{S(\dot{H}^{s-s_0},[0,T])} &\leq O\left(T^\frac{s-s_c}{2}\|u_0\|_{\dot{H}^s})\right)^{\alpha}\|D^{s_0}(u-v)\|_{S(\dot{H}^{s-s_0},[0,T])}.
\end{align}
We then choose $T = \delta\|u_0\|_{\dot{H}^s}^{-\frac{2}{s-s_c}}$ for $\delta >0$ small enough (depending only on universal constants). {
The continuous dependence in this case is similar to the inhomogeneous case, using Lemma \ref{lem_guz2} instead of Lemma \ref{lem_guz}, and it will be omitted.
}
\end{proof}

From the nonlinear estimates, one can also prove a blow-up rate for subcritical finite-time blow up.
\begin{cor}[Subcritical blow-up rate] \label{BUrate}
Under the conditions of Theorem \ref{HomTheo} (a), if the corresponding solution $u$ blows up in finite positive time $T>0$, then
\begin{equation}
    \|u(t)\|_{\dot{H}^s} \gtrsim \frac{1}{(T-t)^{\frac{s-s_c}{2}}},
\end{equation}
for any $t \in [0,T)$.
\end{cor}
\begin{proof} We first observe that the local theory (in the subcritical case) implies
\begin{equation}\label{Alternative}
\|u(t)\|_{\dot{H}^s} \to \infty, \mbox{ as } t\nearrow T.
\end{equation}
Now, assume by contradiction that there exists a sequence $t_n\nearrow T$ such that, for all $n$,
\begin{equation}\label{bound_0}
    (T-t_n)^{\frac{s-s_c}{2}}\|u(t_n)\|_{\dot{H}^s} < \frac{1}{n}.
\end{equation}
From the Duhamel formula and arguing as in the proof of Theorem \ref{HomTheo} (a), for $t \in [t_n, T)$, there exists $C_0>0$ such that
\begin{equation}\label{ContArg}
    \|D^{s}u\|_{\tilde{S}(L^2,[t_n,t])} \leq  
    C_0\|u(t_n)\|_{\dot{H}^s} + \frac{C_0}{n^\alpha \|u(t_n)\|_{\dot{H}^s}^\alpha} \|D^{s}u\|_{{S}(L^2,[t_n,t])}^{\alpha+1}.
\end{equation}

Consider the function $f(x)=x-C_0\|u(t_n)\|_{\dot{H}^s} + \frac{C_0}{n^\alpha \|u(t_n)\|_{\dot{H}^s}^\alpha} x^{\alpha+1}$. A simple computation revels that it has a global maximum at $$x_n^{\ast}=\frac{n\|u(t_n)\|_{\dot{H}^s}}{[C_0(\alpha+1)]^{1/\alpha}},\quad f(x_n^{\ast})=\left(\frac{\alpha n}{(\alpha+1)[C_0(\alpha+1)]^{1/\alpha}}-C_0\right)\|u(t_n)\|_{\dot{H}^s}$$
 and $f(0)<0$. Thus, for $n_0> \frac{\left[C_0(\alpha+1)\right]^{\frac{\alpha+1}{\alpha}}}{\alpha}$ we have that $f(x_{n_0}^{\ast})>0$.

Setting $x(t)=\|D^{s}u\|_{{S}(L^2,[t_{n_0},t])}$, from \eqref{ContArg} we have that $f(x(t))\leq 0$, for all $t \in [t_{n_0}, T)$. A continuity argument then implies that $x(t)\leq x_{n_0}^{\ast}$ and therefore
\begin{equation}\label{bound_1}
     \|D^{s}u\|_{{S}(L^2,[t_{n_0},T))} \lesssim \|u(t_{n_0})\|_{\dot{H}^s}. 
\end{equation}
Estimate \eqref{ContArg} now implies $u \in L^{\infty}_t\dot{H}^s_x([0,T])$, which is a contradiction with \eqref{Alternative}.
    
\end{proof}

To prove Theorem \ref{HomTheo} (b), we rely on the following stability theory.

\begin{lem}[Critical stability] \label{Ls<1}
Let $0\leq s \leq 1$, $s < N/2$,  $\alpha = \frac{4-2b}{N-2s}$, $u_0 \in {H}^{s}$ and $\tilde{u}$ be a $C^0_tH^{s}_x$ solution to 
\begin{equation}
    i\partial \tilde{u} + \Delta \tilde{u} +\mu |x|^{-b}|\tilde{u}|^{\alpha}\tilde{u} = e,
\end{equation}
where $\mu=\pm 1$. Assume that the boundedness conditions
\begin{equation}
    \|D^{s}\tilde{u}\|_{\tilde{S}(L^2,I)} \leq E, \quad \|D^{s} e\|_{S'(L^2,I)} \leq E
\end{equation}
and the smallness conditions
\begin{equation}
    \|e^{it\Delta}(u_0-\tilde{u}_0)\|_{S(\dot{H}^{s},I)} \leq \epsilon, \quad \|e\|_{S'(\dot{H}^{-s},I)} \leq \epsilon,
\end{equation}
hold for $0< \epsilon < \epsilon_0 = \epsilon_0(E)$. Then there is a solution $u$ to \eqref{INLS} with initial datum $u_0$ such that
\begin{equation}
    \|u-\tilde{u}\|_{S(\dot{H}^{s},I)} \lesssim_{E}\epsilon
\end{equation}
and
\begin{equation}
    \|D^{s} u\|_{\tilde{S}(L^2,I)} \lesssim_E 1.
\end{equation}
\end{lem}
\begin{proof}
Write $w = u - \tilde{u}$ and split $I$ in a finite number $N \sim  E/\delta$ of intervals $I_j = [t_j,t_{j+1}]$ such that $\|\tilde{u}\|_{S(\dot{H}^{s},I_j)} \sim \delta$ for all $j$. Let $$F_j = \left\{w: \|w\|_{S(\dot{H}^{s},I_j)} \leq 2 \|e^{i(t-t_j)\Delta} w(t_j)\|_{S(\dot{H}^{s},I_j)}+2C\epsilon, \|D^{s} w\|_{\tilde{S}(L^2,I_j)} \leq 2C \|w(t_j)\|_{\dot{H}^{s}}+2CE 
\right\}$$ be the (complete) metric space with the metric

$$
\rho_j(w_1,w_2) = \|w_1-w_2\|_{S(\dot{H}^{s},I_j)}
$$
and define
\begin{align}
    \Phi_j(w) &= e^{i(t-t_j)\Delta}w(t_j) +i\mu \int_{t_j}^t e^{i(t-\tau  )\Delta} \left[|x|^{-b}|w+\tilde{u}|^{\alpha}(w(\tau)  +\tilde{u}(\tau)  )-|x|^{-b}|\tilde{u}|^{\alpha}\tilde{u}(\tau)  \right] \, d\tau  \\&\quad\quad- i\mu \int_{t_j}^t e^{i(t-\tau  )\Delta} e(\tau)   \, d\tau .
\end{align}
Then,
\begin{align}
    \|\Phi_j(w)\|_{S(\dot{H}^{s},I_j)} &\leq \|e^{i(t-t_j)\Delta}w(t_j)\|_{S(\dot{H}^{s},I_j)}\\
    &\quad+ C\left[\|w\|_{S(\dot{H}^{s},I_j)}+\|\tilde{u}\|_{S(\dot{H}^{s},I_j)}\right]^{\alpha}\|w\|_{S(\dot{H}^{s},I_j)}+C\epsilon \\
    &\leq 
    [1+o_{\delta,\epsilon}(1)]
    \|e^{it\Delta}w(t_j)\|_{S(\dot{H}^{s},I_j)}+
    [1+o_{\delta,\epsilon}(1)]C\epsilon,\\
    \|D^{s} \Phi_j(w)\|_{\tilde{S}(L^2,I_j)} &\leq C \|w(t_j)\|_{\dot{H}^{s}}
    +C\|w+\tilde{u}\|_{S(\dot{H}^{s},I_j)}^{\alpha}(\|D^{s}w\|_{S(L^2,I_j)}+\|D^{s}\tilde{u}\|_{S(L^2,I_j)})\\
    &\quad+C\|\tilde{u}\|_{S(\dot{H}^{s},I_j)}^{\alpha}\|D^{s}\tilde{u}\|_{S(L^2,I_j)}+CE\\
    &\leq[1+o_{\delta,\epsilon}(1)]C\|w(t_j)\|_{\dot{H}^{s}}+[1+o_{\delta,\epsilon}(1)]C
    E,\\
    \|\Phi_j(w_2)-\Phi_j(w_1)\|_{S(\dot{H}^{s},I_j)} &\leq C\left[\|w_2\|_{S(\dot{H}^{s},I_j)}^{\alpha}+\|w_1\|_{S(\dot{H}^{s},I_j)}^{\alpha}+\|\tilde{u}\|_{S(\dot{H}^{s},I_j)}^{\alpha}\right]\|w_2-w_1\|_{S(\dot{H}^{s},I_j)}\\
    &\leq 
    o_{\delta,\epsilon}(1)\|w_2-w_1\|_{S(\dot{H}^{s},I_j)}.
\end{align}
It is clear that we can impose $\delta$ and $\epsilon$ small enough, in such a way that both depend only on universal constants, concluding that each $\Phi_j$ is a contraction on $F_j$. By inductively constructing the solution on $I_0, I_1, \cdots, I_{N-1}$, we have a solution $w$ defined on $I$ such that, for all $j>0$,

\begin{align}
    \|D^{s} w\|_{\tilde{S}(L^2,I_j)} &\leq 2C \|w(t_j)\|_{\dot{H}^{s}} + 2CE,\label{recursion_1}\\
    \|w\|_{S(\dot{H}^{s},I_j)} &\leq 2\|e^{i(t-t_j)\Delta} w(t_j)\|_{S(\dot{H}^{s},I_j)}+2C\epsilon\label{recursion_2}.
\end{align}
Noting that $\|w(t_j)\|_{\dot{H}^{s}} \leq \|D^{s} w\|_{\tilde{S}(L^2,I_{j-1})}$, we get, by \eqref{recursion_1}, $\|D^{s} w\|_{\tilde{S}(L^2,I)}\lesssim C^N E \sim C^{\frac{E}{\delta}}E$. Now, observing, by Duhamel, that

\begin{align}
    e^{i(t-t_{j})\Delta } w(t_j)&= e^{i(t-t_{j-1}\Delta)}w(t_{j-1})+i\mu\int_{t_{j-1}}^{t_j}e^{i(t-\tau  )\Delta}\left[|x|^{-b}|w+\tilde{u}|^{\alpha}(w(\tau)  +\tilde{u}(\tau)  )\hspace{-1pt}-\hspace{-1pt}|x|^{-b}|\tilde{u}|^{\alpha}\tilde{u}(\tau)  \right] \, d\tau  \\
    &\quad\quad- i\mu \int_{t_{j-1}}^{t_j} e^{i(t-\tau  )\Delta} e(\tau)   \, d\tau ,
\end{align}
we get
\begin{equation}
    \|e^{i(t-t_{j})\Delta } w(t_j)\|_{S(\dot{H}^{s})} \leq 2\|e^{i(t-t_{j-1})\Delta } w(t_{j-1})\|_{S(\dot{H}^{s})}+2C\epsilon.
\end{equation}
Therefore, by \eqref{recursion_2}, 
\begin{equation}
    \|w\|_{S(\dot{H}^{s},I)} \lesssim 2^N C \epsilon \sim 2^{\frac{E}{\delta}} C \epsilon.
\end{equation}
\end{proof}

\begin{proof}[Proof of Theorem \ref{HomTheo}-(b)] Let $\{u_{0,n}\}_n$ be a sequence of functions in $H^s$ such that $\|u_{n,0}-u_0\|_{\dot{H}^s} \to 0$ as $n \to +\infty$. Let $T_0>0$ be such that 
\begin{equation}\label{small_linear_evol}
   { \|e^{it\Delta}u_0\|_{S(\dot{H}^s,[0,T_0])}<\delta_0},
\end{equation}
for a small $\delta_0>0$ to be chosen later. Note that, if $n$ is large enough, then
\begin{equation}
    \|e^{it\Delta}u_{n,0}\|_{S(\dot{H}^s,[0,T_0])}\lesssim \delta_0.
\end{equation}

By Duhamel, Strichartz and Lemma \ref{lem_guz3}, one then has, for all $t \in [0,T_0]$:

\begin{equation}
    \|u_n\|_{S(\dot{H}^s,[0,t])} \lesssim \delta_0 + \|u_n\|_{S(\dot{H}^s,[0,t])}^{\alpha+1}.
\end{equation}

This bootstraps to $\|u_n\|_{S(\dot{H}^s,[0,T_0])} \lesssim \delta_0$, if $\delta_0$ is small (depending only on universal constants). We then have
\begin{equation}
   { \|D^su_n\|_{\tilde{S}(L^2,[0,T_0])} 
   \lesssim \|u_0\|_{\dot{H}^s}+\|u_n\|_{S(\dot{H}^s,[0,T_0])}^\alpha\|D^su_n\|_{\tilde{S}(L^2,[0,T_0])},}
\end{equation}
which finally implies $\|D^su_n\|_{\tilde{S}(L^2,[0,T_0])} \lesssim \|u_0\|_{\dot{H}^s}$.

Thus, we can employ Lemma \ref{Ls<1} (with $e \equiv 0$) to guarantee that the sequence of solutions $\{u_n\}_n$ is Cauchy in the norm $\| \cdot \|_{S(\dot{H}^s,[0,T_0])}$, therefore converging to, say, $u$. By reflexivity, uniqueness of strong and weak limits and lower semicontinuity of the norm, passing to a subsequence, we can assume that $u \in L^{\infty}_t\dot{H}^s_x([0,T_0]) \cap S^s(L^2,[0,T_0])$. Strichartz estimates again let us write, for $[t_1,t_2]\subset [0,T_0]$:
\begin{equation}
    \|u(t_2)-u(t_1)\|_{\dot{H}^s}\lesssim \|[e^{i(t_2-t_1)\Delta}-I]u_0\|_{\dot{H}^s}+\|u\|_{S(\dot{H}^s,[t_1,t_2])}^{\alpha}\|D^su\|_{S(L^2,[0,T_0])}, 
\end{equation}
which shows that $u \in C^{0}_t\dot{H}^s_x([0,T_0])$. Uniqueness also follows from Strichartz: if $u$ and $v$ satisfy the Duhamel formula with initial datum $u_0$, then \eqref{small_linear_evol}, Strichartz and Lemma \ref{lem_guz3} imply
\begin{equation}
    \|u\|_{S(\dot{H}^s,[0,T_0])}+\|v\|_{S(\dot{H}^s,[0,T_0])} \lesssim \delta_0.
\end{equation}
Estimating the difference, we get
\begin{equation}
    \|u-v\|_{S(\dot{H}^s,[0,T_0])} \lesssim o_{\delta_0}(1) \|u-v\|_{S(\dot{H}^s,[0,T_0])},
\end{equation}
which implies $u \equiv v$. Similarly to the well-posedness result in $H^s$, we remark that if 
\begin{equation}\label{small_data_dot}
   { \|e^{it\Delta}u_0\|_{S(\dot{H}^s,\mathbb{R})}<\delta_0},
\end{equation}
then the solution is defined for all $t \in \mathbb R$. 

{The continuous dependence on the initial data also relies on the stability theory. Indeed, let $E>0$ and $u_0 \in \dot H^s$ be such that $\|u_0\|_{\dot H^s} \leq E$. If $T_0$ is chosen so that \eqref{small_linear_evol} holds for $\delta_0(E) >0$ small enough, and $v_0 \in \dot H^s$ is such that $\|u_0-v_0\|_{\dot H^s} < \epsilon \ll_{E} \delta_0$, then the corresponding solutions $u$ and $v$ satisfy, by Lemma \ref{Ls<1},
\begin{align}
    \|D^s u\|_{S(L^2,[0,T_0])} + \|D^s v\|_{S(L^2,[0,T_0])} &\lesssim_E 1,\\
    \|u\|_{S(\dot H ^s,[0,T_0])} + \|v\|_{S(\dot H ^s,[0,T_0])} &\lesssim_E \delta_0
    \\
    \|u-v\|_{S(\dot H ^s,[0,T_0])} \lesssim_E \epsilon. 
\end{align}

We can use Duhamel and Strichartz, together with Lemma \ref{lem_guz3} to write, in the case $\alpha \geq 1$,
\begin{align}
    \|D^{s}(u-v)\|_{\tilde S (L^2,[0,T_0])} \lesssim_E \epsilon +  (\delta_0^\alpha + \delta_0^{\alpha-1})   \|D^{s}(u-v)\|_{\tilde S (L^2,[0,T_0])}, 
\end{align}
which implies $\|D^{s}(u-v)\|_{\tilde S (L^2,[0,T_0])} \lesssim_E \epsilon$. In the case $\alpha < 1$, the estimate is 
\begin{equation}
    \|D^{s}(u-v)\|_{\tilde S (L^2,[0,T_0])} \lesssim_E \epsilon +  \delta_0^\alpha   \|D^{s}(u-v)\|_{\tilde S (L^2,[0,T_0])} + \epsilon^{\alpha}, 
\end{equation}
which, in turn, readily implies
    \begin{equation}
    \|D^{s}(u-v)\|_{\tilde S (L^2,[0,T_0])} \lesssim_E \epsilon + \epsilon^{\alpha}. 
\end{equation}
}
\end{proof}

\section{Ill-posedness}\label{Ill}

We now turn to the proof of Theorem \ref{thm:illposed}. We will prove that, if $u=e^{it\Delta}u_0$, $u_0$ smooth, the integral term in Duhamel's formula becomes singular at the $H^s$-level. In order to extract the precise singular behavior, the main idea is to use Lipschitz regularity in $x$ to replace $|u(t,x)|^\alpha u(t,x)$ by $|u(t,0)|^\alpha u(t,0)$. The resulting Duhamel term then involves the linear evolution of the weight $|x|^{-b-s}$, which can be determined (almost) explicitly.

More precisely, set $f=|u|^\alpha u$. We decompose
\begin{align*}
D^s\int_0^t e^{i(t-\tau  )\Delta}|x|^{-b}f(\tau  ,x)d\tau  &= \int_0^t e^{i(t-\tau  )\Delta}\left[D^s\left(|x|^{-b}f(\tau  ,x)\right) - \left(D^s|x|^{-b}\right)f(\tau  ,x)\right]d\tau  \\&+ \int_0^t e^{i(t-\tau  )\Delta}\left[\left(D^s|x|^{-b}\right)\left(f(\tau  ,x)-f(\tau  ,0)\right)\right]d\tau  \\&+ \int_0^te^{i(t-\tau  )\Delta}\left(D^s|x|^{-b}\right)f(\tau  ,0)d\tau  =: \mathbf{I} + \mathbf{II} + \mathbf{III}.
\end{align*}

Throughout this section, let $s_0\gg \max \{N/2+1,N\}$ be fixed. We begin by recalling the following classical result (see H\"{o}rmander \cite{hormander}*{Chapter 7}):
\begin{lem}\label{lem:hormander}
	If $0<b<N$ and $s>0$ satisfies $b+s - N\notin \mathbb{N}\cup\{0\}$, then
	$$
	D^s|x|^{-b}=c_{N,b,s}|x|^{-b-s} \,\, \textrm{in} \,\, \mathcal{S}'(\mathbb{R}^N).
	$$
\end{lem}

\begin{lem}\label{lem:controlo_II}
	If $b+s<2+N/2$,
	$$
	\|\mathbf{II}\|_{S'(L^2, I)}\lesssim_{|I|} \|u\|_{L^\infty(I;H^{s_0})}^{\alpha+1}.
	$$
\end{lem}
\begin{proof}
	We split the estimate in the regions $B=B_1(0)$ and $B^c$. 

\noindent	\textit{Step 1. Control near the origin.} In $B$, we make use of the Lipschitz estimate
\begin{equation}
	|f(t  ,x)-f(t  ,0)|\lesssim |x|\|f(t)\|_{W^{1,\frac{N}{2}^+}_x}.
\end{equation}
For $(q,r)$ admissible with $r$ close to $2N/(N-2)^+$,
\begin{align*}
	\left\|\int_0^t e^{i(t-\tau  )\Delta}\left[\left(D^s|x|^{-b}\right)\left(f(\tau  ,x)-f(\tau  ,0)\right)\mathbbm{1}_{B}\right]d\tau \right\|_{\tilde{S}(L^2, I)} &\lesssim \|(D^s|x|^{-b})(f(t  ,x)-f(t  ,0))\|_{L^{q'}_tL_x^{r'}(B)} \\&\lesssim \||x|^{-b-s+1}\|f(t)\|_{W^{1,\frac{N}{2}^+}_x}\|_{L^{r'}_x(B)}\\&\lesssim\||x|^{-b-s+1}\|_{L^{r'}_x(B)}\|u\|_{L^{\infty}_tH^{s_0}_x}^{\alpha+1} \\&\lesssim \|u\|_{L^{\infty}_tH^{s_0}_x}^{\alpha+1}.
\end{align*}
Notice that the integrability of $|x|^{-b-s+1}$ is ensured by the condition $b+s<2+N/2$.

\noindent\textit{Step 2. Control away from the origin.} Since $b+s>N/2$,
\begin{align*}
	\left\|\int_0^t e^{i(t-\tau  )\Delta}\left[\left(D^s|x|^{-b}\right)\left(f(\tau  ,x)-f(\tau  ,0)\right)\mathbbm{1}_{B^c}\right]d\tau \right\|_{\tilde{S}(L^2, I)} &\lesssim \||x|^{-b-s}\|f(t)\|_{L^\infty_x}\|_{L^{1}_tL^{2}_x(B^c)}\\&\lesssim\||x|^{-b-s}\|_{L^2_x(B^c)}\|f\|_{L^1_tL^\infty_x} \\&\lesssim \|u\|_{L^{\infty}_tH_x^{s_0}}^{\alpha+1}.
\end{align*}
\end{proof}

\begin{lem}\label{lem:controlo_I}
	If $b<\min\{N,2\}$,
	$$
	\|\mathbf{I}\|_{S'(L^2, I)}\lesssim_{|I|} \|u\|_{L^\infty(I;H^{s_0})}^{\alpha+1}.
	$$
\end{lem}
\begin{proof}
	For $r$ close to $2N/(N-2)^+$, take $r_2$ such that
	$$
	\frac{1}{r'}=\frac{b}{N}+\frac{1}{r_2}
	$$
	(observe that $(N+2)/2N>b/N$, which ensures that $r_2\in(1,\infty)$).
	Applying Lemma \ref{gen_leib_1},
	\begin{align*}
		\|D^s(|x|^{-b}f)-(D^s|x|^{-b})f\|_{L^{r'}_x}&\lesssim \||x|^{-b}\|_{L^{(N/b)^-}(B_1)}\|D^s f\|_{L^{r_2^+}} +\||x|^{-b}\|_{L^{(N/b)^+}(B_1^c)}\|D^s f\|_{L^{r_2^-}} \\&\lesssim \|u\|_{L^{\infty}_tH_x^{s_0}}^{\alpha+1}.
	\end{align*}
 Therefore, by Strichartz estimates,
$$
	\|\mathbf{I}\|_{S'(L^2, I)} \lesssim 	\|D^s(|x|^{-b}f)-(D^s|x|^{-b})f\|_{L^{q'}_tL^{r'}_x} \lesssim \|u\|_{L^{\infty}_tH_x^{s_0}}^{\alpha+1}.
$$ 
\end{proof}

Now that we have reduced the analysis to the term $\mathbf{III}$, define
$$
(Tf)(t,x)=\int_0^t f(t-\tau  )e^{i\tau\Delta}D^s|x|^{-b}d\tau,
$$
so that
$$
\mathbf{III}=T\left(|u|^\alpha u\Big|_{x=0}\right).
$$
By Lemma \ref{lem:hormander},
$$
(Tf)(t,x)=c_{b,s,N}\int_0^t f(t-\tau  )e^{i\tau\Delta}|x|^{-b-s}d\tau.
$$
To derive the precise asymptotics for $T$, we need an accurate description of $e^{it\Delta}|x|^{-b-s}$. For the rest of this section, we set
\begin{equation}
	\theta=\frac{b+s}{2}, \quad \beta=\frac{N-b-s}{2}.
\end{equation}

We start with the explicitly expression obtained by Cazenave \cite{cazenave} in the case $0<b+s<N$.

\begin{lem}[{\cite{cazenave}*{Lemmas 2.6.2 and 2.6.4}}]\label{lem:cazenave}
For $0<b+s<N$,
\begin{equation}
	e^{it\Delta}|x|^{-b-s}=\frac{1}{(4it)^{\frac{b+s}{2}}}\frac{1}{\Gamma(\theta)}H\left(\frac{|x|^2}{4t}\right),
\end{equation}
where $H\in C^\infty(\mathbb{R}^N)$ is defined as
\begin{equation}
	H(y;\theta,\beta)=\int_0^1e^{iyr}r^{\theta-1}(1-r)^{\beta-1}dr.
\end{equation}
Furthermore, one has the following asymptotics for $H$
\begin{equation}\label{eq:Hasympt}
	H(y;\theta,\beta)= \Gamma(\theta)e^{i\frac{\pi\theta}{2}} y^{-\theta} + \Gamma(\beta)e^{-i\frac{\pi\beta}{2}}e^{iy}y^{-\beta} + O(y^{-\theta-1}) + O(y^{-\beta-1})\quad \mbox{as }y\to \infty.
\end{equation}
\end{lem}

For the case $b+s>N$, we need to suitably extend the formula above.
\begin{lem}\label{lem:extensaoH}
	Given $\max\{N,2\}<b+s<N+2$ such that $b+s\neq N+1$, one has
	\begin{equation}\label{eq:extensaoH}
		e^{it\Delta}|x|^{-b-s}=\frac{1}{\beta(4it)^{\frac{b+s}{2}}\Gamma((b+s)/2)}\left[(\theta-1)H\left(\frac{|x|^2}{t};\theta-1,\beta+1\right)+\frac{i|x|^2}{4t}H\left(\frac{|x|^2}{t},\theta,\beta+1\right)\right].
	\end{equation}
	As a consequence, for $|x|^2\gg t$,
	\begin{equation}\label{eq:asymptmaiorN}
		e^{it\Delta}|x|^{-b-s} \sim \frac{1}{t^{\theta}}\left[e^{i\frac{|x|^2}{4t}}\left(\frac{|x|^2}{4t}\right)^{-\beta} + O\left(\left(\frac{|x|^2}{t}\right)^{-\theta}\right) + O\left(\left(\frac{|x|^2}{t}\right)^{-\beta-1}\right)\right].
	\end{equation}
\end{lem}
\begin{proof}
	Fix a test function $\varphi\in C^\infty_c(\mathbb{R}^N)$. From Lemma \ref{lem:cazenave}, for $0<\lambda<N$,
	\begin{equation}
		\left\langle e^{it\Delta}|x|^{-\lambda}, \varphi\right\rangle =\left\langle \frac{1}{(4it)^{\frac{\lambda}{2}}}\frac{1}{\Gamma(\lambda/2)}H\left(\frac{|x|^2}{4t};\frac{\lambda}{2}, \frac{N-\lambda}{2}\right),\varphi\right\rangle.
	\end{equation}
	For $2<\lambda<N$,
	\begin{align*}
		\frac{N-\lambda}{2} H\left(y;\frac{\lambda}{2},\frac{N-\lambda}{2}\right)&=\frac{N-\lambda}{2}\int_0^1 e^{iyr}r^{(\lambda-2)/2}(1-r)^{(N-\lambda-2)/2}dr\\&=\int_0^1 \partial_r\left(e^{iyr}r^{(\lambda-2)/2}\right)(1-r)^{(N-\lambda)/2}dr \\&=  iyH\left(y;\frac{\lambda}{2},\frac{N+2-\lambda}{2}\right) + \frac{\lambda-2}{2}H\left(y;\frac{\lambda-2}{2},\frac{N+2-\lambda}{2}\right) 
	\end{align*}
	Therefore
	\begin{align}
		\frac{N-\lambda}{2}&\left\langle e^{it\Delta}|x|^{-\lambda}, \varphi\right\rangle \\&=\left\langle \frac{1}{(4it)^{\frac{\lambda}{2}}}\frac{1}{\Gamma(\lambda/2)}\left[\frac{\lambda-2}{2}H\left(\frac{|x|^2}{t};\frac{\lambda-2}{2},\frac{N+2-\lambda}{2}\right)+\frac{i|x|^2}{4t}H\left(\frac{|x|^2}{t},\frac{\lambda}{2},\frac{N+2-\lambda}{2}\right)\right] ,\varphi\right\rangle.
	\end{align}
	
	By Gel\cprime~fand and Shilov \cite{gelfandshilov}*{Section III.3.2}, the l.h.s. is meromorphic in $\lambda$, with poles at the integers greater or equal than $N$. On the other hand, the r.h.s. is analytic for $2<\lambda<N+2$. By unique continuation, \eqref{eq:extensaoH} follows. Using \eqref{eq:Hasympt}, for large $y$,
	\begin{align*}
		(\theta-1)H(y;\theta-1,\beta+1)+iyH(y;\theta,\beta+1)&=(\theta-1)\Gamma(\theta-1)e^{\frac{i\pi(\theta-1)}{2}}y^{-\theta+1} + iy\Gamma(\theta)e^{\frac{i\pi \theta}{2}}y^{-\theta}\\ &\quad + \Gamma(\beta+1)e^{-\frac{i\pi(\beta+1)}{2}}e^{iy}y^{-\beta-1}\left((\theta-1)+iy\right)\\&\quad  + O(y^{-\theta}) + O(y^{-\beta-2}) \\&=i\Gamma(\beta+1)e^{-\frac{i\pi(\beta+1)}{2}}e^{iy}y^{-\beta} + O(y^{-\theta}) + O(y^{-\beta-1}).
	\end{align*}
The precise cancellation of the terms in $y^{-\theta+1}$ gives the claimed asymptotic expansion \eqref{eq:asymptmaiorN}.
\end{proof}

\begin{re}
	As we are considering the case $b+s>\min\{N,N/2+1\}$, we have $\beta<\theta$. In particular, regardless of the relation between $b+s$ and $N$, the asymptotic behavior of $e^{it\Delta}|x|^{-b-s}$ is given by
	\begin{equation}\label{eq:asympt}
		e^{it\Delta}|x|^{-b-s} \sim \frac{1}{t^{\theta}}\left[e^{i\frac{|x|^2}{4t}}\left(\frac{|x|^2}{4t}\right)^{-\beta} + O\left(\left(\frac{|x|^2}{t}\right)^{-\eta}\right) \right],\quad \eta=\min\{\theta,\beta+1\}.
	\end{equation}
\end{re}

\medskip

As the dynamical behavior of $e^{it\Delta}|x|^{-b-s}$ depends on the quocient $|x|^2/t$, we further decompose
\begin{equation}
	\mathbf{III} = \int_0^tf(t-\tau  )\mathbbm{1}_{|x|^2<  \tau}e^{i\tau\Delta}\left(D^s|x|^{-b}\right)d\tau  + \int_0^tf(t-\tau  )\mathbbm{1}_{|x|^2> \tau}e^{i\tau\Delta}\left(D^s|x|^{-b}\right)d\tau  = \mathbf{III}_1 + \mathbf{III}_2.
\end{equation}
\begin{lem}\label{lem:controlo_III1}
Under the assumptions of either Lemma \ref{lem:cazenave} or \ref{lem:extensaoH}, if $b+s<N/2+2$,	
$$\|\mathbf{III}_1\|_{L^\infty_tL^2_x} \lesssim_{|I|} \|u\|_{L^\infty(I;H^{s_0})}^{\alpha+1} $$ 
\end{lem}

\begin{proof}
	By duality,
	\begin{align}\label{eq:estIII1}
		&\left|\int_0^T\int \int _0^t f(t-\tau  )\frac{1}{\tau^\theta}H\left(\frac{|x|^2}{\tau},\theta, \beta\right)\mathbbm{1}_{|x|^2<\tau}\phi(x,t)d\tau dxdt\right|\nonumber\\\lesssim& \int_0^T\int_0^t|f(t-\tau  )|\frac{1}{\tau^\theta}\left(\int_{|x|^2<\tau}\phi(x,t)dx\right)d\tau dt\nonumber \\\lesssim& \int_0^T\int_0^t \tau^{\frac{N}{4}-\theta}\|f\|_{L^\infty_t}\|\phi(t)\|_{L^{2}_x}d\tau dt\\\lesssim& \int_0^T\int_0^t \tau^{\frac{N}{4}-\theta}\|u\|_{L^\infty_tH^{s_0}_x}^{\alpha+1}\|\phi(t)\|_{L^{2}_x}d\tau dt \lesssim \|u\|_{L^\infty_tH^{s_0}_x}^{\alpha+1}\|\phi\|_{L^1_tL^{2}_x}
	\end{align}
since $N/4-\theta>-1$.
\end{proof}
 Using the asymptotic expansion for $H$, we write
\begin{align*}
	\mathbf{III}_2=\int_0^tf(t-\tau  )\mathbbm{1}_{|x|^2> \tau}e^{i\tau\Delta}\left(D^s|x|^{-b}\right)d\tau  &= c\int_0^t f(t-\tau  )\frac{1}{\tau^\theta}\left(\frac{|x|^2}{\tau}\right)^{-\beta}e^{i|x|^2/\tau} \mathbbm{1}_{|x|^2> \tau} d\tau  \\&\quad + \int_0^t f(t-\tau  )\frac{1}{\tau^\theta}O\left(\left(\frac{|x|^2}{\tau}\right)^{-\eta}\right) \mathbbm{1}_{|x|^2> \tau} d\tau \\& =: \mathbf{III}_{21} + \mathbf{III}_{22}.
\end{align*}

\begin{lem}\label{lem:controlo_III22}
Under the assumptions of either Lemma \ref{lem:cazenave} or \ref{lem:extensaoH}, if $b+s<N/2+2$,
$$\|\mathbf{III}_{22}\|_{L^\infty_tL^2_x}\lesssim_{|I|} \|u\|_{L^\infty(I;H^{s_0})}^{\alpha+1}. $$ 

\end{lem}
\begin{proof}
	The proof follows from similar arguments to those in the previous lemma. Taking $\phi\in L^1_tL^2_x$,
\begin{align*}
	&\int_0^T\int \int_0^t \frac{|f(t-\tau  )|}{\tau^{\theta} }\left(\frac{|x|^2}{\tau}\right)^{-\eta}\mathbbm{1}_{|x|^2> \tau }\phi(x,t)d\tau dxdt\\ \lesssim &\left( \int_0^T \frac{1}{\tau^{\theta} }\left\|\left(\frac{|x|^2}{\tau}\right)^{-\eta}\mathbbm{1}_{|x|^2> \tau } \right\|_{L^2_x}d\tau\right) \|\phi\|_{L^1_tL^2_x}\|u\|_{L^\infty(I;H^{s_0})}^{\alpha+1}.
\end{align*}
	Since $4\eta>N$,
	$$
	\int_{\tau\le|x|^2} \frac{\tau^{2\eta}}{|x|^{4\eta}}dx\lesssim \tau^{N/2}
	$$
and, since $N/4-\theta>-1$, the integral in $\tau$ is convergent.
\end{proof}

\begin{lem}\label{lem:explode1}
	Suppose that
	$$
	1+\frac{N}{2}\le b+s<2+\frac{N}{2},\quad b+s\neq N, N+1.
	$$
	Let $u_0(x)=e^{-\pi|x|^2}$ and $u=e^{it\Delta}u_0$. Then, for any $t>0$ fixed,
	$$
\left\|\int_0^tD^se^{i(t-\tau)\Delta}|x|^{-b}|u|^\alpha(\tau) u(\tau)d\tau\right\|_{L^2(\mathbb{R}^N)}=+\infty.
	$$
\end{lem}
\begin{proof}
 We write
	$$
	\int_0^tD^se^{i(t-\tau)\Delta}|x|^{-b}|u|^\alpha(\tau) u(\tau)d\tau = \mathbf{I} + \mathbf{II} + \mathbf{III}_{1} + \mathbf{III}_{22} + \mathbf{III}_{21}
	$$
	By Lemmas \ref{lem:controlo_II}, \ref{lem:controlo_I}, \ref{lem:controlo_III1} and \ref{lem:controlo_III22}, the first four terms have bounded $L^2_x$ norm. Let us show that the last term has infinite $L^2_x$ norm. Take a test function $\phi\in L^2(\mathbb{R}^N)$ supported near $|x|^2\sim 4t$.
	\begin{align}
		\int\mathbf{III}_{21}\phi dx&= \int \int_0^t \frac{f(t-\tau)}{\tau^{\theta-\beta}}|x|^{-2\beta}e^{i|x|^2/\tau}d\tau \phi(x)dx\\&= \int |x|^{-2\beta}\phi(x)\left(\int_0^t \frac{f(t-\tau)}{\tau^{\theta-\beta}}e^{i|x|^2/\tau}d\tau\right)dx \\&= \int |x|^{-2\theta+2}\phi(x)\int_{|x|^2/t}^\infty f\left(t-\frac{|x|^2}{s}\right)s^{\theta-\beta-2}e^{is}ds\\&=\int |x|^{-2\theta+2}\phi(x)\left(\int_{|x|^2/t}^\infty \left[f\left(t-\frac{|x|^2}{s}\right)-f(t)\right]s^{\theta-\beta-2}e^{is}ds\right)dx \\&\qquad + f(t)\int |x|^{-2\theta+2}\phi(x)\left(\int_{|x|^2/t}^\infty s^{\theta-\beta-2}e^{is}ds\right)dx.\label{eq:diverge}
	\end{align}
Since 
$$\theta-\beta-2\geq -1\Leftrightarrow b+s\geq \frac{N}{2}+1$$
and 
$$f(t)=\left(|e^{it\Delta}u_0|^\alpha e^{it\Delta}u_0\right)\Big|_{x=0}=\frac{1}{|1+4i\pi t|^{N\alpha/2}(1+4i\pi t)^{N/2}}\neq 0,
$$
 the second integral diverges. For the first integral, we use the mean-value theorem:
\begin{align}
	&\ \left|\int |x|^{-2\theta+2}\phi(x)\int_{|x|^2/t}^\infty \left[f\left(t-\frac{|x|^2}{s}\right)-f(t)\right]s^{\theta-\beta-2}e^{is}ds\right|\\\lesssim & \ \|f'\|_{L^\infty(0,t)}\int |x|^{-2\theta}\phi(x)\int_{|x|^2/t}^\infty s^{\theta-\beta-3} ds <\infty.\label{eq:meanvalue}
\end{align}
\end{proof}
If $N=1$, $N<\frac{N}{2}+1$ and the previous result does not cover the range
$$
1=N<b+s<\frac{N}{2}+1=\frac{3}{2}.
$$
In this case, we are still able to show 
\begin{lem}\label{lem:explode2}
	For $N=1$, suppose that
	$$
	1< b+s<\frac{3}{2}.
	$$
	Let $u_0(x)=e^{-\pi|x|^2}$ and $u=e^{it\Delta}u_0$. Then, for any $t>0$ fixed,
	$$
	\left\|\int_0^tD^se^{i(t-\tau)\Delta}|x|^{-b}|u|^\alpha(\tau) u(\tau)d\tau\right\|_{L^2(\mathbb{R}^N)}=+\infty.
	$$
\end{lem}
\begin{proof}
	As in the previous proof, it suffices to prove that $\|\mathbf{III}_{21}\|_{L^2(\mathbb{R}^N)}=+\infty$. Fix $\phi\in C^\infty_c(-1,1)$ and set $\phi_k(x)=\phi(x-\sqrt{k})$, $k$ large. Recall that, by \eqref{eq:diverge},
	 	\begin{align*}
	 	\int\mathbf{III}_{21}\phi_k dx&=\int |x|^{-2\theta+2}\phi_k(x)\left(\int_{|x|^2/t}^\infty \left[f\left(t-\frac{|x|^2}{s}\right)-f(t)\right]s^{\theta-\beta-2}e^{is}ds\right)dx \\&\qquad + f(t)\int |x|^{-2\theta+2}\phi_k(x)\left(\int_{|x|^2/t}^\infty s^{\theta-\beta-2}e^{is}ds\right)dx.
	 \end{align*}
 As in \eqref{eq:meanvalue}, the first integral is uniformly bounded in $k$. For the second,
 $$
 \int |x|^{-2\theta+2}\phi_k(x)\left(\int_{|x|^2/t}^\infty s^{\theta-\beta-2}e^{is}ds\right)dx \sim  \int |x|^{-2\beta}\phi_k(x)dx \sim k^{-\beta} \to \infty \mbox{ as }k\to \infty
 $$
 since $\beta<0$.
\end{proof}

%

\begin{proof}[Proof of Theorem \ref{thm:illposed}]
We follow a classical argument, see Tzvetkov \cite{tzvetkov} and Bejenaru and Tao \cite{taobejenaru}. Suppose that there exists a $C^{\alpha+1}$-flow defined in $H^s(\mathbb{R}^N)$. Given $\delta>0$ and $T>0$, let $u(\delta)\in C([0,T], H^s(\mathbb{R}^N))$ be the solution with initial data $\delta u_0$, for $u_0= e^{-\pi|x|^2}$. Then $u(0)=0$,
$$
\frac{\partial u(\delta)}{\partial \delta}\Big|_{\delta=0}=e^{it\Delta}u_0,\qquad \frac{\partial^{(j)} u(\delta)}{\partial \delta^{(j)}}\Big|_{\delta=0}=0,\quad j=2,\dots,\alpha
$$
and
$$
\frac{\partial^{(\alpha+1)} u(\delta)}{\partial \delta^{(\alpha+1)}}\Big|_{\delta=0}=i\mu\int_0^te^{i(t-\tau)\Delta}|x|^{-b}|e^{i\tau\Delta}u_0|^{\alpha}e^{i\tau\Delta}u_0 d\tau.
$$
By Lemmas \ref{lem:explode1} and \ref{lem:explode2}, the last term is not in $H^s(\mathbb{R}^N)$, for any $t>0$, contradicting the smoothness of the flow.
\end{proof}

%
%
%
%



\end{document}